\documentclass{amsart}
\usepackage{amssymb, amsmath, latexsym, mathabx, dsfont,tikz}

\renewcommand{\baselinestretch}{\baselinestretch}
\renewcommand{\baselinestretch}{1.1}
\numberwithin{equation}{section}

\newtheorem{thm}{Theorem}[section]
\newtheorem{lem}[thm]{Lemma}

\newtheorem{prop}[thm]{Proposition}

\theoremstyle{definition}
\newtheorem{defn}[thm]{Definition}

\theoremstyle{remark}
\newtheorem{rmk}[thm]{Remark}
\newtheorem{exam}[thm]{Example}

\numberwithin{equation}{section}

\newcommand{\gen}{\text{gen}}
\newcommand{\spn}{\text{spn}}
\newcommand{\ord}{\text{ord}}
\newcommand{\diag}{\text{diag}}
\newcommand{\z}{{\mathbb Z}}
\newcommand{\q}{{\mathbb Q}}
\newcommand{\f}{{\mathbb F}}

\newcommand{\p}{{\mathbb P}}

\newcommand{\T}{{\tilde{T}}}
\newcommand{\gr}{{\mathfrak{G}_{L,p}(m)}}

\begin{document}

\title[A generalization of Watson transformation]{A generalization of Watson transformation and representations of ternary quadratic forms}

\author{Jangwon Ju, Inhwan Lee and Byeong-Kweon Oh}

\address{Department of Mathematical Sciences, Seoul National University, Seoul 151-747, Korea}
\email{jjw@snu.ac.kr}

\address{Department of Mathematical Sciences, Seoul National University, Seoul 151-747, Korea}
\email{lih0905@snu.ac.kr}

\address{Department of Mathematical Sciences and Research Institute of Mathematics, Seoul National University, Seoul 151-747, Korea}
\email{bkoh@snu.ac.kr}

\thanks{This research was supported by Basic Science Research Program through the National Research Foundation of Korea(NRF) funded by the Ministry of Education(NRF-2014R1A1A2056296).}



\begin{abstract} Let $L$ be a positive definite (non-classic) ternary $\z$-lattice  and let $p$ be a prime such that a $\frac 12\z_p$-modular component of $L_p$  is nonzero isotropic and $4\cdot dL$ is not divisible by $p$.  For a nonnegative integer $m$, let $\mathcal G_{L,p}(m)$ be the genus with discriminant $p^m\cdot dL$ on the quadratic space $L^{p^m}\otimes \q$ such that for each lattice $T \in \mathcal G_{L,p}(m)$,   a $\frac 12\z_p$-modular component of $T_p$  is nonzero isotropic, and $T_q$ is isometric to $(L^{p^m})_q$ for any prime $q$ different from $p$. Let $r(n,M)$ be the number of representations of an integer $n$ by a $\z$-lattice $M$.  In this article, we show that  if $m \le 2$ and $n$ is divisible by $p$ only when $m=2$, then  for any $T \in \mathcal G_{L,p}(m)$, $r(n,T)$  can be written as a linear summation of $r(pn,S_i)$ and $r(p^3n,S_i)$ for $S_i \in  \mathcal G_{L,p}(m+1)$ with an extra term in some special case.  We provide a simple criterion on when the extra term is necessary, and we compute the extra term explicitly.  We also give a recursive relation to compute $r(n,T)$, for any $T \in \mathcal G_{L,p}(m)$, by using the number of representations of some integers by lattices in $\mathcal G_{L,p}(m+1)$ for an arbitrary integer $m$.    
\end{abstract}

\maketitle

\section{Introduction} 

For a positive definite (non-classic) integral ternary quadratic form 
$$
f(x_1,x_2,x_3)=\sum_{1 \le i\le j\le 3}a_{ij}x_ix_j \qquad (a_{ij} \in \z)  
$$ 
and an integer $n$, we define a set $R(n,f)=\{ (x_1,x_2,x_3) \in \z^3 : f(x_1,x_2,x_3)=n\}$, and 
 $r(n,f)=\vert R(n,f)\vert$. It is well known that $R(n,f)$ is always  finite  if $f$ is positive definite.
The theta series $\theta_f(z)$ of $f$ is defined by 
$$
\theta_f(z)=\sum_{n=0}^{\infty}r(n,f)e^{2\pi inz},
$$ 
which is a modular form of weight $\frac32$ and some character with respect to a certain congruence subgroup. Finding a closed formula for $r(n,f)$ or finding all integers $n$ such that $r(n,f) \ne 0$ for an arbitrary ternary form $f$ are quite old problems which are still widely open. As a simplest case, Gauss showed that if $f$ is a sum of three squares, then $r(n,f)$ is a multiple of the Hurwitz-Kronecker class number. 
 
 Though it seems to be quite difficult to find a closed formula for $r(n,f)$, some various relations between $r(n,f)$'s are known. One of the important relations is the Minkowski-Siegel formula.      
Let $O(f)$ be the group of isometries of $f$ and $o(f)=\vert O(f)\vert$. The weight $w(f)$ of $f$ is defined by 
$w(f)=\sum_{ [f'] \in \text{gen}(f)} \frac 1{o(f')}$,  where $[f']$ is the equivalence class containing $f'$. The Minkowski-Siegel formula says that the weighted sum of the representations by quadratic forms in the genus is, in principle, the product of local densities, that is,   
$$
\frac 1{w(f)}\sum_{[f'] \in \text{gen}(f)} \frac {r(n,f')}{o(f')}=c^*\prod_{p} \alpha_p(n,f_p),
$$
where the constant $c^*$ can easily be computable and $\alpha_p$ is the local density depending only on the local structure of $f$ over $\z_p$. Hence if the class number of $f$ is one, then we have a closed formula on $r(n,f)$. As a natural modification of the    Minkowski-Siegel formula, it was proved in \cite {k} and \cite {w} that the weighted sum of the representations of quadratic forms in the spinor genus is also equal to the product of local densities except spinor exceptional integers (see also \cite {sp3} for spinor exceptional integers).

For any prime $p \nmid 2df$, the action of Hecke operators $T(p^2)$ on the theta series of the quadratic form $f$ gives  
$$
r(p^2n,f)+\left(\frac{-ndf}{p}\right)r(n,f)+p\cdot r\left(\frac n{p^2},f\right)=\sum_{[f'] \in \gen(f)} \frac {r^*(p^2f',f)}{o(f')} r(n,f').
$$
Here, if $n$ is not divisible by $p^2$, then $r\left(\frac n{p^2},f\right)=0$, and $r^*(p^2f',f)$ is the number of primitive representations of $p^2f'$ by $f$.  For details, see \cite {a} and \cite{e}.

Another important relation comes from the Watson transformation. If a unimodular component of the ternary form $f$ in a Jordan decomposition  over $\z_p$ is anisotropic, then one may easily show that 
$$
r(pn,f)=r(pn,\Lambda_p(f)),
$$ 
where $\Lambda_p(f)$ is defined in Section 2. Hence the theta series of $f$ completely determines the theta series of $\lambda_p(f)$. Unfortunately if a unimodular component of the ternary form $f$ over $\z_p$ is isotropic, one cannot expect such a nice relation. In this article, we consider the case when a unimodular component of the ternary form $f$ over $\z_p$ is isotropic.                 

The subsequence discussion will be conducted in the more adapted geometric language of quadratic spaces and lattices.  
The term ``lattice" will always refer to a positive definite non-classic integral $\z$-lattice  on an
$n$-dimensional positive definite quadratic space over $\q$. Here, a $\z$-lattice is said to be  {\it non-classic}  if the norm ideal $\mathfrak n(L)$ of $L$ is contained in $\z$.   
Let $L=\z x_1+\z x_2+\cdots +\z x_n$ be a $\z$-lattice of rank $n$.
 We write
 $$
 L \simeq (B(x_i,x_j)).
$$
The right hand side matrix is called a {\it matrix presentation} of
$L$. Any unexplained notations and terminologies can be found in
\cite{ki} or \cite{om2}.

Let $V$ be a (positive definite) ternary quadratic space and
let $L$ be a (non-classic) ternary $\z$-lattice on $V$. Let $p$ be a prime such that  $L_p \simeq \begin{pmatrix} 0&\frac12\\ \frac12&0\end{pmatrix} \perp \langle \epsilon \rangle$, where $\epsilon \in \z_p^{\times}$. For any nonnegative integer $m$, let $\mathcal G_{L,p}(m)$ be a genus on a quadratic space $W$ such that each $\z$-lattice $T \in \mathcal G_{L,p}(m)$ satisfies
$$
 T_p \simeq  \begin{pmatrix} 0&\frac12\\ \frac12&0\end{pmatrix} \perp \langle \epsilon p^m \rangle \quad \text{and} \quad  T_q \simeq (L^{p^m})_q \ \text{ for any $q \ne p$}.
$$
Here $W=V$ if $m$ is even, $W=V^p$ otherwise.  
The aim of this article is to show that if $T\in  \mathcal G_{L,p}(m)$ for $m =0$ or $1$, then there are rational numbers $a_i, b_i$ such that
$$
r(n,T)=\sum_{[S_i] \in \mathcal G_{L,p}(m+1)} \left(a_ir(pn,S_i)+b_ir(p^3n,S_i)\right)+(\text{some extra term}).
$$
In Section 4, we prove this statement in each case and compute the rational numbers $a_i$'s, $b_i$'s and the extra term explicitly.  For the case when $m=2$, we give an example such that the above statement does not hold, and prove that the above statement still holds for $m=2$ if we additionally assume that $n$ is divisible by $p$. In the case when
 $m\ge 3$, we show that under some restriction, the above statement holds if we replace $r(n,T)$ by $r(p^2n,T)-pr(n,T)$, and for any integer $n$ not divisible by $p$, both $r(n,T)$ and $r(pn,T)$ can be written as a linear summation of  $r(pn, S)$'s  and $r(n,S)$'s, respectively,  for $S \in \mathcal G_{L,p}(m+1)$.   

 In some cases, the extra term in the above equation can be removed. To determine when it happens, we need to know some structure of the graph $\mathfrak G_{L,p}(m)$
 defined by the equivalence classes in $\mathcal G_{L,p}(m)$ and $\mathcal G_{L,p}(m+1)$. The definition and basic facts on the graph  $\mathfrak G_{L,p}(m)$ will be treated in Section 3.      
	
For any integer $a$, we say that $\frac a2$ is divisible by a prme $p$ if $p$ is odd and $a \equiv 0 \pmod p$, or $p=2$ and $a \equiv 0 \pmod 4$.  	

\section{A generalization of Watson transformation}

Let $L$ be a ternary $\z$-lattice. Recall that we are assuming that a (quadratic) $\z$-lattice is  non-classic and positive definite. For any prime $p$, the $\lambda_p$-transformation (or Watson transformation) is defined as follows: 
$$
\Lambda_p(L)= \{ x \in L : Q(x + z) \equiv Q(z) \  (\text{mod} \ p) \mbox{ for
all $z \in L$}\}.
$$
 Let $\lambda_p(L)$ be the primitive
lattice obtained from $\Lambda_p(L)$ by scaling $V=L\otimes \mathbb Q$ by a suitable
rational number. Assume that $p$ is odd. If the unimodular component
in a Jordan decomposition of $L_p$ is anisotropic, it is well known that
\begin{equation}
R(pn,L)=R(pn,\Lambda_p(L)).
\end{equation}
Hence $r(n,\lambda_p(L))=r(pn,L)$ if $p\z_p$-modular component of $L_p$ is nonzero, and  $r(n,\lambda_p(L))=r(p^2n,L)$ otherwise. One may easily show that (2.1) still holds for $p=2$ unless
$$
L_2 \simeq \begin{pmatrix} 0&\frac12\\ \frac12&0\end{pmatrix} \perp \langle \alpha\rangle, \ \ (\alpha \in \z_2).
$$
The readers are referred to \cite{co0} for more properties of the operators $\Lambda_p$.

Let $L$ be a ternary $\z$-lattice and let $p$ be a fixed prime. In the remaining of this section, we always assume that  in a Jordan splitting of $L_p$,
\begin{equation}
\text {\it the $\frac12 \z_p$-modular component is non-zero isotropic.}
\end{equation} 
The purpose of this article is to find similar results to (2.1) under this assumption. To do this, we generalize Watson's transformation in various directions. Since 
$$
\begin{pmatrix} 1&\frac12\\ \frac12&1\end{pmatrix} \perp \langle \delta \rangle \simeq \begin{pmatrix} 0&\frac12\\ \frac12&0\end{pmatrix} \perp \langle 5\delta \rangle \ \ \text{over $\z_2$}
$$
for any $\delta \in \z_2^{\times}$, any $\z$-lattice $L$ such that $L_2$ is isometric to the above will also be considered when $p=2$.

\begin{defn}  \label{defn1} Assume that $p$ is odd. For $\epsilon=0$ or $\pm1$, we define
$$
S_p(\epsilon,L)=\left\lbrace x \in L ~\Bigg{|}~ \left(\frac {Q(x)}p\right)=\epsilon\ \right\rbrace.
$$
We also define $S_2(0,L)=\{ x \in L : Q(x)\equiv 0 \pmod 2\}$ and $S_2(*,L)=L-S_2(0,L)$.   
\end{defn}

Let $\mathfrak B=\{x_1,x_2,x_3\}$ be a (ordered) basis of a ternary $\z$-lattice $L$ and $p$ be a prime. We define a natural projection map  
$$
\phi_{\mathfrak B}: L-pL \to (L/pL)^{*} \to \mathbb P^2,
$$
 where $\mathbb P^2$ is the $2$-dimensional projective space over the  finite field $\mathbb F_p$. The set $\phi_{\mathfrak B}(S_p(\epsilon,L)-pL)$ is denoted by
 $s_p^{\mathfrak B}(\epsilon,L)$ for any $\epsilon \in \{ 0, 1, -1\}$ if $p$ is odd and $\epsilon \in \{ 0,*\}$ otherwise. If the basis $\mathfrak B$ is obvious, we will omit it.  
For each element $\mathbf s \in \p^2$, we define a $\z$-sublattice $L_{\mathbf s}:=\phi_{\mathfrak B}^{-1}({\mathbf s}) \cup pL$ of $L$, and
$$
\Omega_p(\epsilon,L)=\{ L_{\mathbf s} \mid \mathbf s \in  s_p^{\mathfrak B}(\epsilon,L)\}.
$$
Note that if  $T : \mathfrak B \to \mathfrak C$ is the transition  matrix between ordered bases, then one may easily show
that $T(s_p^{\mathfrak B}(\epsilon,L))=s_p^{\mathfrak C}(\epsilon,L)$. Hence the set $\Omega_p(\epsilon,L)$ is independent of  choices of the basis for $L$.

\begin{lem} \label{1} Assume that a ternary $\z$-lattice  $L$ and a prime $p$ satisfies the condition (2.2).  
If $4dL_p \in \z_p^{\times}$, then 
$$
\vert s_p(0,L)\vert=p+1, \ \ \vert s_p(\pm1,L)\vert=\frac {p\left(p\pm\left(\frac {-dL}p\right)\right)}2 \quad \text{and} \quad s_2(*,L) = 4
$$
and
$$
\vert s_p(0,L)\vert=2p+1, \ \ \vert s_p(1,L)\vert=\vert s_p(-1,L)\vert=\frac {p(p-1)}2 \quad \text{and} \quad s_2(*,L) = 2,
$$
otherwise. 
\end{lem}

\begin{proof} Since everything is trivial for $p=2$, we assume that $p$ is odd. 
For the unimodular case, see Theorem 1.3.2 of \cite{ki}. Assume that $L_p$ is not unimodular. 
Fix an ordered basis $\mathfrak B=\{x_1,x_2,x_3\}$ of $L$ such that 
$$
(B(x_i,x_j)) \equiv \diag(1,-1,p^{\ord_p(dL)}\delta) \pmod {p^{\ord_p(dL)+1}},
$$
for some $\delta \in \z-p\z$. Note that such a basis always exists by the Weak Approximation Theorem. 
Assume that $x=a_1x_1+a_2x_2+a_3x_3 \in S_p(0,L)$. Then $a_1^2 \equiv a_2^2 \pmod p$. Therefore 
$$
s_p^{\mathfrak B}(0,L)=\{ (0,0,1),(1,\pm 1,d)\}, \qquad \text{where } d \in \mathbb F_p.
$$
The lemma follows from this. The case when $\epsilon=\pm 1$ can be done in a similar manner.    
\end{proof}

\begin{lem} \label{local} Under the same assumptions given above, assume that $p$ is an odd prime. 
If $\epsilon \ne 0$ or $\epsilon=0$ and $L_p$ is unimodular, then every $\z$-lattice $M \in \Omega_p(\epsilon,L)$ is contained in one genus. 
Furthermore for the former case, 
$$
M_q \simeq \begin{cases} \langle \delta,-p^2\delta,-p^2dL\rangle\quad &\text{if $q=p$,} \\
                                 L_q   &\text{otherwise,} \end{cases}
$$
where $\delta \in \z_p^{\times}$ such that $\left(\frac \delta p\right)=\epsilon$ and, 
$$
M_q \simeq \begin{cases} \langle p,-p,-p^2dL\rangle\quad &\text{if $q=p$,} \\
                                  L_q   &\text{otherwise,} \end{cases}
$$
for the latter case. If $L_p$ is not unimodular and $\epsilon=0$ then  every $\z$-lattice $M \in \Omega_p(0,L)$ is 
exactly contained in two genera. More precisely
$$
M_q \simeq \begin{cases} \langle p^2,-p^2,-dL\rangle \ \ \text{or} \ \  \langle p,-p,-p^2dL\rangle \quad &\text{if $q=p$,} \\
                                 L_q   &\text{otherwise.} \end{cases}
$$
 \end{lem}

\begin{proof} Let $L=\z x_1+\z x_2+\z x_3$ and $M \in \Omega_p(\epsilon,L)$. 
Since $pL \subset M$, we may assume without loss of generality that $M=\z(x_1+b_2x_2+b_3x_3)+\z(px_2)+\z(px_3)$. 
First assume that $\epsilon \ne0$. Then we may further assume that $\left(\frac {Q(x_1+b_2x_2+b_3x_3)}p\right)=\epsilon$. 
Since  $Q(x_1+b_2x_2+b_3x_3) \in \z_p^{\times}$, 
$$
M_p \simeq \langle Q(x_1+b_2x_2+b_3x_3) \rangle \perp m_p
$$ 
for some binary  sublattice $m_p$ of $M_p$ whose scale is $p^2\z_p$. 
The assertion follows from this. Assume that $\epsilon=0$ and $L_p$ is unimodular. In this case we may assume that 
 $Q(x_1+b_2x_2+b_3x_3) \in p\z_p$. Then $B(x_1+b_2x_2+b_3x_3,x_2)$ or  $B(x_1+b_2x_2+b_3x_3,x_3)$ is a unit in $\z_p$, for 
$L_p$ is unimodular. The assertion follows from this. 

Finally assume that $L_p$ is not unimodular and $\epsilon=0$. In this case we may assume that the ordered basis $\mathfrak B=\{x_1,x_2,x_3\}$
satisfies every condition in Lemma \ref{1}. Then by a direct computation we know $L_{(0,0,1)} \in \Omega_p(0,L)$ satisfies the first local property and the others satisfy the second local property.  
\end{proof}

\begin{lem} \label{local-2} Under the same assumptions given above, assume that $p=2$. Let $M$ be a $\z$-lattice in $\Omega_2(\epsilon,L)$. If $-4dL_2=\delta \in \z_2^{\times}$, then 
$$
M_2 \simeq \begin{cases}  \begin{pmatrix} 0&1\\1&0\end{pmatrix} \perp \langle 4\delta\rangle \qquad &\text{ if $\epsilon=0$},\\
\langle 1,-1,4\delta\rangle \quad \text{or} \quad  \begin{pmatrix} 0&2\\2&0\end{pmatrix} \perp \langle \delta\rangle  \qquad  &\text{ otherwise}, \end{cases}
$$
and $M_q \simeq L_q$ for any prime $q \ne 2$.
 If $-4dL_2=\delta \in 2\z_2$, then 
$$
M_2 \simeq \begin{cases}  \begin{pmatrix} 0&1\\1&0\end{pmatrix} \perp \langle 4\delta\rangle 
\quad \text{or} \quad  \begin{pmatrix} 0&2\\2&0\end{pmatrix} \perp \langle \delta\rangle
\qquad &\text{ if $\epsilon=0$},\\
\langle 1,-1,4\delta\rangle   \qquad  &\text{otherwise}, \end{cases}
$$       
and $M_q \simeq L_q$ for any prime $q \ne 2$.

 \end{lem}

\begin{proof}
The proof is quite similar to the above. 
\end{proof}

\begin{lem} \label{2} Assume that a ternary $\z$-lattice  $L$ and a prime $p$ satisfies the condition (2.2).  
 For any positive integer $n$ such that $\left( \frac np\right)=\epsilon$,
$$
r(n,L)=\sum_{M \in \Omega_p(\epsilon,L)} r(n,M)-(\vert s_p(\epsilon,L)\vert-1)r(n,pL).
$$
This equality also holds for $p=2$ if either $\epsilon = 0$ and $n$ is even or $\epsilon = *$ and $n$ is odd.

\end{lem}

\begin{proof} 
The lemma follows from the facts that
$$
\{ x \in S_p(\epsilon,L)-pL \mid Q(x)=n, \ \ \phi(x)=s \}=\{ x \in L_s \mid Q(x)=n\}-R(n,pL),
$$  
and 
$$
L_s \cap L_t=pL \qquad \text{if and only if} \qquad s\ne t,
$$ 
for any $s,t \in \p^2$.
\end{proof}

Under the same assumptions given above, one may easily show that $dM=p^4dL$ for any $M \in \Omega_p(\epsilon,L)$. Furthermore
$L/M \simeq \z/p\z\oplus \z/p\z$. 

\begin{rmk} {\rm If  a $\frac 12 \z_p$-modular component of $L_p$ is zero or anisotropic, the above lemma implies the equation (2.1). So we may consider the above lemma as a natural generalization of Watson's transformation.}
\end{rmk}

Let $L$ and $\ell$ be  ternary $\z$-lattices such that $d\ell=p^4dL$. We define
 $$
\tilde{R}(\ell,L)=\{ \sigma : \ell \to L \mid L/\sigma(\ell) \simeq \z/p\z\oplus \z/p\z\} \ \  \text{and}  \ \ \tilde{r}(\ell,L)=\vert \tilde{R}(\ell,L)\vert.
$$
One may easily show that $\vert\{ M \in \Omega_p(\epsilon,L) \mid M \simeq \ell \} \vert=\tilde{r}(\ell, L)/o(\ell)$ for any $\epsilon \in \{0,\pm1\}$ or $\epsilon \in \{0,*\}$.

\begin{lem} \label{tilde r} For any ternary $\z$-lattices $\ell$ and $L$ such that $d\ell=p^4dL$, we have
$$
\tilde{r}(\ell,L)=r(p\ell^{\#},L^{\#})=r(pL,\ell). 
$$
\end{lem}

\begin{proof} Assume that $T \in \tilde{R}(\ell,L)$. Then $T^tM_LT=M_{\ell}$ and $pT^{-1}$ is an integral
matrix. Since 
$$
(pT^{-1})M_L^{-1}(pT^{-1})^t=p^2M_{\ell}^{-1},
$$ 
$(pT^{-1})^t \in R(p\ell^{\#},L^{\#})$. Conversely if $S^tM_L^{-1}S=p^2M_{\ell}^{-1}$, then $d(S)=\pm p$. 
Hence $pS^{-1}$ is an integral matrix and $(pS^{-1})^t \in \tilde{R}(\ell,L)$. This completes the proof.
\end{proof}

Assume that a ternary $\z$-lattice $L$ and a prime $p$ satisfies the condition (2.2). In the remaining of this section, 
we additionally assume that $\ord_p(4\cdot dL) \ge 2$.    
Let $K=\lambda_p(L)$ and let
$$
\gen_p^K(L)=\{ L' \in \gen(L) : \lambda_p(L') \simeq K\}.
$$
For any integer $n$, we also define 
$$
r(n, \gen_p^K(L)) = \sum_{\substack{[L']\in\gen(L) \\ \lambda_p(L') \simeq K}} \frac{r(n,L')}{o(L')}.
$$ 
 In fact, every $\z$-lattice in $\gen_p^K(L)$ is isometric to one of $\z$-lattices in
$$
\Gamma_p^L(\Lambda_p(L))=\{ M \subset K \mid M \in \gen(L)\}.
$$ 
Furthermore, the isometry group $O(K)$ acts on $\Gamma_p^L(\Lambda_p(L))$. Each orbit under this action consists of all isometric lattices in $\Gamma_p^L(\Lambda_p(L))$, and hence there are exactly $\frac {o(K)}{o(L)}$ lattices  that are isometric to $L$ in $\Gamma_p^L(\Lambda_p(L))$.  There are exactly $p^2+p+1$ sublattices
of $K$ with index $p$. They are, in fact,   
$$ 
K_0=\z (px_1)+\z x_2+\z x_3, \ \  \ K_{1,u}=\z(x_1+ux_2)+\z(px_2)+\z x_3 \ (0 \le u \le p-1) 
$$
and
$$
K_{2,\alpha,\beta}=\z(x_1+\alpha x_3)+\z(x_2+\beta x_3)+\z(px_3) \ (0 \le \alpha,\beta \le p-1).
$$ 
Among these sublattices of $K$, there are exactly $\frac {p(p+1)}2$ lattices ($p^2$ lattices) that are contained in the genus of $L$ if 
$\ord_p(4\cdot dL)=2$ ($\ord_p(4\cdot dL) \ge 3$, respectively) (for details, see \cite {co}).

\begin{prop} \label{genTL-}  Assume that $\z$-lattices $L$ and $K$ and a prime $p$ satisfies the above condition.  Then for any integer $n$ not divisible by $p$,  we have 
$$
r(n,\gen_p^K(L))=\begin{cases} \displaystyle \frac{p-\left(\frac{-ndK}{p}\right)}2\frac{r(n,K)}{o(K)} \quad &\text{if 
 $p\ne 2$ and $\ord_p(4\cdot dL)=2$}, \\
 \displaystyle \frac {r(n,K)-r(n,\Lambda_1(K))}{o(K)}  \quad &\text{if 
 $p=2$ and $\ord_p(4\cdot dL)=2$}, \\
\displaystyle p\frac{r(n,K)}{o(K)} \quad &\text{if $\ord_p(4\cdot dL)\ge 3$}, \\ \end{cases}
$$
where $\Lambda_1(K)=\{ x \in K : B(x,K) \subset \z\}$ is a sublattice of $K$. 
\end{prop}

\begin{proof} Since proofs are quite similar to each other, we only provide the proof of the first case. Assume that $Q(x_1)=n$ for some $x_1 \in K$. We will count the number of lattices  containing the vector $x_1$ in $\Gamma_p^L(\Lambda_p(L))$. Note that for any vector $y \in K$ and any integer $d$ not divisible by $p$, $dy \in M$ if and only if $y \in M$ for any $M \in \Gamma_p^L(\Lambda_p(L))$. Hence we may assume that $x_1$ is a primitive vector in $K$.  
Then there is a basis $\{x_1,x_2,x_3\}$ of $K$ such that for some integer $t$ not divisible by $p$,
$$
(B(x_i,x_j))\equiv \diag(n,n,t) \pmod{p}.
$$
Among all sublattices of $K$ with index $p$ that are contained in the genus of $L$, those $\z$-lattices containing $x_1$ are 
 $K_{2,0,\beta}$, for any $\beta$ satisfying $\left(\frac {-n^2-n\beta^2 dK}{p}\right)=1$, and  
 $K_{1,0}$ only when $\left(\frac {-n dK}p\right)=1$. 
Therefore one may easily show that the total number of such lattices  is  
$\frac {p-\left(\frac{-ndK}{p}\right)}2$.  The proposition follows from
$$
\sum_{M \in \Gamma_p^L(\lambda_p(L))}r(n,M)=\sum_{[M] \in \gen_p^K(L)} \frac{o(K)}{o(M)} r(n,M)= \frac {p-\left(\frac{-ndK}{p}\right)}2 r(n,K).
$$
This completes the proof.  
\end{proof} 

\begin{prop} \label{genTL} Under the same assumption given above,
if $n$ is divisible by $p$, then we have
$$
r(n,\gen_p^K(L)) =\begin{cases} p \displaystyle\frac{r(n,K)}{o(K)} + \frac{p(p-1)}{2} \frac{r\left(\frac{n}{p^2} , K\right)}{o(K)} & \text{if $\ord_p(4\cdot dL)=2$}, \\
\displaystyle p \frac{r(n,K)}{o(K)} + p^2 \frac{r\left( \frac{n}{p^2}, K\right)}{o(K)} - p \frac{r(n,\Lambda_p(K))}{o(K)} & \text{otherwise}. \end{cases}
$$
\end{prop}  

\begin{proof}  First we define
$$
R^{*}(n,K)= \{ x \in K \mid Q(x)=n, \text{ $x$ is  primitive as a vector in $K_p$}\}, 
$$
$r^{*}(n,K)=\vert R^{*}(n,K)\vert$, and $r^{\Diamond}(n,K)=r(n,K)-r^{*}(n,K)$.
Let $x_1 \in K$ be a vector such that $Q(x_1)=n$. We will compute the number of lattices containing $x_1$ in $\Gamma_p^L(\Lambda_p(L))$. By the similar reasoning to the above, we may assume that there is a primitive vector $\widetilde{x_1} \in K$ and a nonnegative integer $k$ such that $x_1=p^k\widetilde{x_1}$. 
 If $k>0$, then $x_1$ is contained in all lattices in   $\Gamma_p^L(\Lambda_p(L))$ .

Assume that $k=0$. If  $\ord_p(4\cdot dL)=2$,
then there is a basis $\{x_1,x_2,x_3\}$ of $K$ such that 
$$
(B(x_i,x_j))\equiv \begin{pmatrix} 0&b&0\\b&0&0\\0&0&e\end{pmatrix} \pmod p,
$$
where $2b$ and $e$ are integers not divisible by $p$. Among all sublattices of $K$ with index $p$ that are contained in the genus of $L$, those $\z$-lattices containing $x_1$ are  $K_{2,0,\beta}$ for any $\beta$. Therefore if  $\ord_p(4\cdot dL)=2$,  we have
$$
\begin{array} {lll}
\displaystyle\sum_{[M] \in \gen_p^K(L)} \frac{o(K)}{o(M)} r(n,M)&=p\cdot r^{*}(n,K)+
\displaystyle\frac{p(p+1)}2 r^{\diamond}(n,K)\\
&=p\cdot r(n,K)+\displaystyle\frac{p(p-1)}2 r\left(\frac n{p^2},K\right).\end{array}
$$
Suppose that   $\ord_p(4\cdot dL)\ge3$. If there is a vector $y\in K$ such that $2B(x_1,y)\not\equiv 0\pmod p$, then there are exactly $p$ lattices in $\Gamma_p^L(\Lambda_p(L))$  containing $x_1$. However if $2B(x_1,K)\subset p\z$, then there does not exist a lattice  in $\Gamma_p^L(\Lambda_p(L))$  that contains $x_1$. Note that
$$
\vert\{ x \in R^{*}(n,K) \mid 2B(x,K) \subset p\z\}\vert=r(n,\Lambda_p(K))-r^{\Diamond}(n,K).
$$
Therefore we have
$$
\sum_{[M] \in \gen_p^K(L)} \frac{o(K)}{o(M)} r(n,M)=p(r(n,K)-r(n,\Lambda_p(K)))+p^2\cdot r^{\Diamond}(n,K).\\
$$
This completes the proof. \end{proof}

\section{Finite (multi-) graphs and ternary quadratic forms}

Let $V$ be a (positive definite) ternary quadratic space and
let $L$ be a (non-classic) ternary $\z$-lattice on $V$. Let $p$ be a prime such that  $L_p \simeq \begin{pmatrix} 0&\frac12\\ \frac12&0\end{pmatrix} \perp \langle \epsilon \rangle$, where $\epsilon \in \z_p^{\times}$. For any nonnegative integer $m$, let $\mathcal G_{L,p}(m)$ be a genus on $W$ such that each $\z$-lattice $T \in \mathcal G_{L,p}(m)$ satisfies
$$
 T_p \simeq  \begin{pmatrix} 0&\frac12\\ \frac12&0\end{pmatrix} \perp \langle \epsilon p^m \rangle \quad \text{and} \quad T_q \simeq (L^{p^m})_q \ \text{ for any $q \ne p$}.
$$
Here $W=V$ if $m$ is even, $W=V^p$ otherwise.  

\begin{lem} \label{2-p} Let $T \in \mathcal G_{L,p}(m)$ and  $S \in \mathcal G_{L,p}(m+1)$  be ternary $\z$-lattices. Then we have 
$$
\sum_{[N] \in \mathcal G_{L,p}(m+1)} \frac{\tilde{r}(N^p,T)}{o(N)} = \begin{cases} p+1 &\text{ \ \ if $m=0$,} \\ 2p &\text{ \ \ otherwise}  \end{cases} \quad \text{and} \quad \sum_{[M] \in \mathcal G_{L,p}(m)} \frac{r(M^p,S)}{o(M)} = 2.
$$
\end{lem}

\begin{proof}
 Note that  $\sum_{[N] \in \mathcal G_{L,p}(m+1)} \frac{\tilde{r}(N^p,T)}{o(N)}$ is the number of sublattices $X$ of $T$ such that 
$$
T/X \simeq \z/p\z\oplus \z/p\z \quad \text{and} \quad X^{\frac1p} \in  \mathcal G_{L,p}(m+1).
$$ 
Hence the first equality is a direct consequence of Lemmas \ref{1}, \ref{local} and \ref{local-2}. 
 
  To prove the second equality, it suffices to show that there are exactly two sublattices of $S$ with index $p$ whose norm is $p\z$.  
  By Weak Approximation Theorem, there exists a basis $\{x_1, x_2, x_3 \}$ for $S$ such that
$$
(B(x_i,x_j))\equiv\begin{pmatrix}0&\frac12\\ \frac12&0\end{pmatrix}\perp \langle p^{m+1} \delta\rangle \ (\text{mod} \ p^{m+2}),
$$
where $\delta$ is an integer not divisible by $p$. Then  for the following two sublattices defined by
$$
\aligned
&\Gamma_{p,1}(S) = \z px_1 + \z x_2+ \z x_3, &\Gamma_{p,2}(S) = \z x_1 + \z px_2+ \z x_3, 
\endaligned
$$
one may easily show that $\Gamma_{p,i}(S)^{\frac1p} \in \mathcal G_{L,p}(m)$ for any $i=1,2$. Furthermore, norms of all the other  sublattices of $S$ with index $p$ are not contained in $p\z$. This completes the proof.  
\end{proof}

Now 	we define a multi-graph $\mathfrak{G}_{L,p}(m)$ as follows: the set of vertices in $\gr$ is the set of equivalence classes in  $\mathcal G_{L,p}(m)$, say, $\{[T_1], [T_2], \ldots , [T_h] \}$. The set of edges is exactly the set of equivalence classes in $\mathcal G_{L,p}(m+1)$, say, $ \{[S_1], [S_2],\ldots,[S_k] \}$. For each equivalence class $[S_w] \in\mathcal G_{L,p}(m+1)$, two vertices contained in the edge named by $[S_w]$ are defined by $[\Gamma_{p,1}(S_w)^{\frac 1p}]$ and  $[\Gamma_{p,2}(S_w)^{\frac1p}]$, where the lattice $\Gamma_{p,i}(S_w)^{\frac1p}$ that is defined in  Lemma \ref{2-p} is contained in $\mathcal G_{L,p}(m)$. 
 Note that the graph $\gr$ is, in general, a multi-graph that might have a loop. We define an $h \times k$ integer matrix $\mathfrak M_{L,p}(m)=(m_{ij})$ as follows:
\begin{displaymath} 
m_{ij} = \begin{cases} 2 \quad  &\textrm{if  $[S_j]$ is a loop of the vertex $[T_i]$}, \\ 
 1 \quad &\textrm{if  $[S_j]$ is not a loop of the vertex $[T_i]$, though it contains $[T_i]$}, \\  
 0 \quad &\textrm{otherwise.} \end{cases}
\end{displaymath}
Therefore $\mathfrak M_{L,p}(m)$ is the incidence matrix of $\mathfrak{G}_{L,p}(m)$ if the graph $\gr$ is simple.

For any $\z$-lattice $T \in \mathcal G_{L,p}(m)$, we define 
$$
\Phi_p(T)=\{ S \in \mathcal G_{L,p}(m+1) : \Gamma_{p,i}(S)^{\frac1p}=T \ \text{for some $i=1,2$}\}
$$
and 
$$
\Psi_p(T)=\{ M \in \mathcal G_{L,p}(m+2) : \lambda_p(M)=T \}.
$$ 
Then Lemma \ref{2-p} implies that $\vert \Phi_p(T)\vert=p+1$ if $m=0$, $\vert \Phi_p(T)\vert=2p$ otherwise. 

\begin{lem} \label{n=2} Let $T\in \mathcal G_{L,p}(0) $ and  $S, S' \in \Phi_p(T)\ (S\ne S')$ be ternary $\z$-lattices on $V$ and $V^p$, respectively.
Then  there is a unique $\z$-lattice $M \in \Psi_p(T)$ such that 
$\{ \Gamma_{p,1}(M)^{\frac{1}{p}}, \Gamma_{p,2}(M)^{\frac{1}{p}}\}=\{ S, S' \}$.
\end{lem}

\begin{proof}
For any $S, S' \in \Phi_p(T)$,  we have $pS \subset S'$. Furthermore since $S \ne S'$ and $\ord_p(4dS)=1$, $S' / pS \simeq \z/p\z \oplus \z / p^2\z$. Therefore, there is a basis $x_1, x_2, x_3$ for $S'$ such that 
$$
S'= \z{x_1} + \z{x_2} + \z{x_3} , \ \  pS=\z{x_1} + \z{px_2}+ \z{p^2x_3} 
$$
and 
$$
(B(x_i, x_j)) = \begin{pmatrix}p^2a&pb&d \\ pb&pc&e \\ d&e&f \end{pmatrix},
$$
where $a,c,f \in \z$, $b,d,e \in \frac{1}{2} \z$ and $p \nmid 2d$.
Define a $\z$-lattice  
$$
M=\left(\z\left({\frac {x_1}p}\right) + \z{x_2} + \z{x_3}\right)^p \in \mathcal G_{L,p}(2).
$$ 
Then one may easily show that  $\lambda_p(M) = T$ and $\{\Gamma_{p,1}(M)^{\frac1p}, \Gamma_{p,2}(M)^{\frac1p}\} = \{S, S'\}$.  
As pointed out earlier, the number of $\z$-lattices $M' \in \mathcal G_{L,p}(2)$ such that $\lambda_p(M')=T$
for any $T \in \mathcal G_{L,p}(0)$ is $\frac{p(p+1)}2$. Furthermore for any such a $\z$-lattice $M'$, we have $\Gamma_{p,i}(M')^{\frac1p} \in \Phi_p(T)$ for any $i=1,2$ and $\vert \Phi_p(T)\vert=p+1$. Now the uniqueness of $M$ follows from this observation. 
\end{proof}

The above lemma says that if $T \in \mathcal G_{L,p}(0)$, then there is always an edge containing $[S]$ and $[S']$ for any $S,S' \in \Phi_p(T)$. However this is not true in general if $T \in \mathcal G_{L,p}(m)$ for a positive integer $m$.

\begin{lem} \label{n=3}
For a positive integer $m$, let $T \in \mathcal G_{L,p}(m)$ and $S, S' \in \Phi_p(T)$ be ternary $\z$-lattices on $V$ and $V^p$, respectively.
If
$$ 
\lambda_p(S)=\Gamma_{p,1}(T)^{\frac1p} \quad \text{and} \quad \lambda_p(S')=\Gamma_{p,2}(T)^{\frac1p}, 
$$  
then there is a unique $\z$-lattice $M \in \Psi_p(T)$ such that  $\{ \Gamma_{p,1}(M)^{\frac1p}, \Gamma_{p,2}(M)^{\frac1p} \}= \{S, S' \}.$
\end{lem}

\begin{proof}
By Weak Approximation Theorem, there is a basis $x_1, x_2, x_3$ for $T$ such that 
$$
(B(x_i, x_j)) \equiv \begin{pmatrix} 0&\frac12 \\ \frac12&0 \end{pmatrix} \perp \langle p^m \delta \rangle ~ \pmod{p^{m+1}},   
$$
where $\delta$ is an integer not divisible by $p$. We may assume that 
$$ 
\Gamma_{p,1}(T)^{\frac1p}=(\z{px_1}+\z{x_2} + \z{x_3})^{\frac1p}, \ \ \Gamma_{p,2}(T)^{\frac1p} =(\z{x_1}+ \z{px_2} + \z{x_3})^{\frac1p}.
$$
One may easily check that 
$$
\aligned
\Phi_p(T) = \{M_{*,\beta}=&(\z{px_1} + \z{(x_2+\beta x_3)} + \z{px_3})^{\frac1p}: 0 \leq \beta \leq p-1 \} \\
 &\ \ \cup \{ M_{\alpha, *}=(\z{(x_1 + \alpha x_3)} + \z{px_2}+ \z{px_3})^{\frac1p}: 0 \leq \alpha \leq p-1 \}\\
 \endaligned
$$
and
$$ 
\Psi_p(T) = \{ M_{\alpha, \beta}=\z{(x_1+ \alpha x_3)} + \z{(x_2+ \beta x_3)} + \z{px_3} : 0 \leq \alpha, \beta \leq p-1 \}. \\
$$
Since $\lambda_p(M_{*,\beta}) =\Gamma_{p,1}(T)^{\frac1p}$ and $\lambda_p(M_{\alpha,*}) =\Gamma_{p,2}(T)^{\frac1p}$ for any $0 \le \alpha, \beta \le p- 1$, there are $\tau, \eta$ such that
$S=M_{*,\tau}$ and $S'=M_{\eta,*}$.
\begin{center}
\begin{tikzpicture}[node distance = 3.0cm, auto]
\node(T) {$\Gamma_{p,1}(T)^{\frac1p}$};
\node(S) [above right of=T] {$T$};
\node(T') [below right of=S] {$\Gamma_{p,2}(T)^{\frac1p}$};
\node(M) [above left of=S] {$S=M_{*,\tau}$};
\node(L) [above right of=M] {$M_{\eta,\tau}$};
\node(M') [below right of=L] {$S'=M_{\eta,*}$};

\draw[-] (T) to node {} (S);
\draw[-] (S) to node {} (T');
\draw[-] (M) to node {} (S);
\draw[-] (M') to node {} (S);
\draw[-] (L) to node {} (M);
\draw[-] (M') to node {} (L);

\draw[->, dashed] (M) to node {$\lambda_p$} (T);
\draw[->, dashed] (L) to node {$\lambda_p$} (S);
\draw[->, dashed] (M') to node {$\lambda_p$} (T');

\end{tikzpicture}
\end{center}
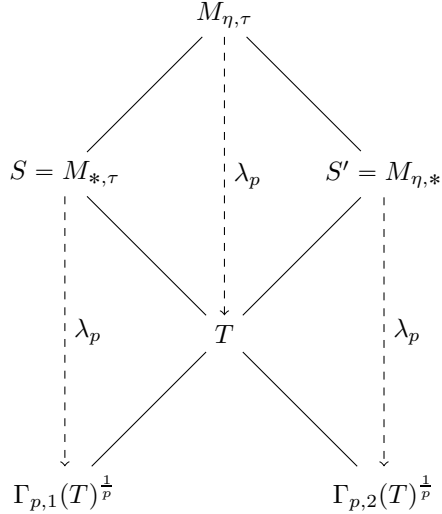
\begin{center}{\bf 3.1 Figure}
\end{center}
 Now, one may easily check that $M_{\eta, \tau}$ is  the unique lattice in $\Psi_p(T)$ satisfying
$$
\{ \Gamma_{p,1}(M_{\eta,\tau})^{\frac1p}, \Gamma_{p,2}(M_{\eta,\tau})^{\frac1p} \}= \{M_{*,\tau}, M_{\eta,*} \}.
$$
This completes the proof. \end{proof}

\begin{lem} \label{nice} For an integer $m\ge 2$, let $M_1, M_2 \in \mathcal G_{L,p}(m)$ be distinct $\z$-lattices such that $\lambda_p(M_1)=\lambda_p(M_2)=T$. Then there is a path from $[M_1]$ to $[M_2]$ of length  $4$.   
\end{lem}

\begin{proof} Note that if $\{ \Gamma_{p,1}(M_1), \Gamma_{p,2}(M_1)\}=\{ \Gamma_{p,1}(M_2), \Gamma_{p,2}(M_2)\}$, then $M_1=M_2$. Hence, without loss of generality, we may assume that $S_1=\Gamma_{p,1}(M_1)^{\frac1p}$ is different from 
 $S_2=\Gamma_{p,2}(M_2)^{\frac1p}$. If $m\ge 3$, then
$$
\{ \lambda_p(\Gamma_{p,1}(M_i)^{\frac1p}),   \lambda_p(\Gamma_{p,2}(M_i)^{\frac 1p})\}=\{ \Gamma_{p,1}(T)^{\frac1p}, \Gamma_{p,2}(T)^{\frac1p} \}
$$ 
 for any $i=1,2$. Hence we further assume that $\lambda_p(S_1) \ne \lambda_p(S_2)$. 
 Then by Lemmas \ref{n=2} and \ref{n=3}, there is a $\z$-lattice $M \in \mathcal G_{L,p}(m)$ such that 
$\lambda_p(M)=T$ and
$\{ \Gamma_{p,1}(M)^{\frac{1}{p}}, \Gamma_{p,2}(M)^{\frac{1}{p}}\}=\{ S_1, S_2 \}$. 
We define $\z$-lattices $T_1$ and $T_2$ satisfying 
$$
\{ \Gamma_{p,1}(S_1)^{\frac{1}{p}}, \Gamma_{p,2}(S_1)^{\frac{1}{p}}\}=\{ T, T_1 \} \quad \text{and} \quad 
\{ \Gamma_{p,1}(S_2)^{\frac{1}{p}}, \Gamma_{p,2}(S_2)^{\frac{1}{p}}\}=\{ T, T_2 \}.
$$
Let $M'_i \in \mathcal G_{L,p}(m)$ be a $\z$-lattice in $\Phi_p(S_i)$ such that $\lambda_p(M'_i)=T_i$ for $i=1, 2$.
Then by Lemma \ref{n=3}, there are $\z$-lattices $N_1, N_2, N'_1, N'_2$ such that 
two vertices $[M_i]$ and $[M'_i]$ are connected by the edge $[N_i]$, and  two vertices $[M]$ and $[M'_i]$ are connected by the edge $[N'_i]$ for $i=1,2$. Therefore two vertices $[M_1]$ and $[M_2]$ are connected by a path of length $4$ (see Figure 3.2).
\begin{center}
\begin{tikzpicture}[node distance = 2.0cm, auto]
\node (T1) {$T_1$};
\node(S1) [above right of=T1] {$S_1$};
\node(T) [below right of = S1] {$T$};
\node(S2) [above right of = T] {$S_2$};
\node(T2) [below right of = S2] {$T_2 $};
\node(M1') [above left of =S1] {$M_1'$};
\node(N1) [above left of =M1'] {$N_1$};
\node(M1) [below left of =N1] {$M_1$};
\node(N1') [above right of = M1'] {$N_1'$};
\node(M) [above right of =S1] {$M$};
\node(N2) [above right of = M] {$N_2'$};
\node(M2') [below right of =N2] {$M_2'$};
\node(N2')[ above right of = M2'] {$N_2$};
\node(M2) [below right of =N2'] {$M_2$};
\draw[-] (T1) to node {} (S1);
\draw[-] (S1) to node {} (T);
\draw[-] (T) to node {} (S2);
\draw[-] (S2) to node {} (T2);
\draw[-] (M1) to node {} (S1);
\draw[-] (M1') to node {} (S1);
\draw[-] (M) to node {} (S1);
\draw[-] (M) to node {} (S2);
\draw[-] (M2') to node {} (S2);
\draw[-] (M2) to node {} (S2);
\draw[-] (M1) to node {} (N1);
\draw[-] (N1) to node {} (M1');
\draw[-] (M1') to node {} (N1');
\draw[-] (N1') to node {} (M);
\draw[-] (M) to node {} (N2);
\draw[-] (N2) to node {} (M2');
\draw[-] (M2') to node {} (N2');
\draw[-] (N2') to node {} (M2);
\draw[->, dashed] (M1') to node {} (T1);
\draw[->, dashed] (M1) to node {} (T);
\draw[->, dashed] (M) to node {} (T);
\draw[->, dashed] (M2') to node {} (T2);
\draw[->, dashed] (M2) to node {} (T);
\draw[->, bend right=15, dashed] (N1) to node {} (S1);
\draw[->, dashed] (N1') to node {} (S1);
\draw[->, dashed] (N2) to node {} (S2);
\draw[->, bend left=15, dashed] (N2') to node {} (S2);
\end{tikzpicture}
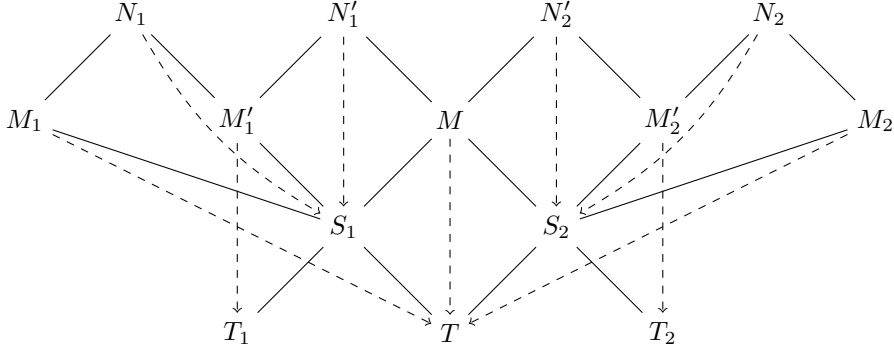
\begin{center}{\bf 3.2 Figure}
\end{center}
\end{center}
The Lemma follows from this.  \end{proof}

\begin{lem} \label{connect} For an integer $m\ge 2$, let $[M], [M']$ be vertices of the graph $\mathfrak G_{L,p}(m)$. Then there is a path from $[M]$ to $[M']$ of length $e([M],[M'])$ in $\mathfrak G_{L,p}(m)$ if and only if  there is a path from $[\lambda_p(M)]$ to $[\lambda_p(M')]$ of length  $e([\lambda_p(M)], [\lambda_p(M')])$ in $\mathfrak G_{L,p}(m-2)$. Furthermore, in both cases,
there is a path satisfying
$$
e([M],[M']) \equiv e([\lambda_p(M)], [\lambda_p(M')]) \pmod 2.
$$
\end{lem} 

\begin{proof} Note that ``only if'' part is trivial.  Assume that $[\lambda_p(M)]$ and $[\lambda_p(M')]$ are connected by a path with edges $[S_1],[S_2],\dots,[S_k]$ as in Figure 3.3, where 
$$
\{ \Gamma_{p,1}(S_i)^{\frac1p}, \Gamma_{p,2}(S_i)^{\frac1p}\}=\{ T_{i-1},T_{i}\}
$$ 
for any $i=2,3,\dots,k-1$.  

\begin{center}
\begin{tikzpicture}[auto]
\node(lL) {$\lambda_p(M)$};
\node(em) [above of =lL, node distance=2.82cm]{};
\node(S1) [above right of=lL, node distance = 2.0cm] {$S_1$};
\node(T1) [below right of=S1, node distance = 2.0cm] {$T_1$};
\node(L0) [right of=em, node distance = 0.4cm] {$M_0$};
\node(L) [left of =em, node distance = 0.4cm]{$M_{}$};

\node(L1) [above right of=S1, node distance = 2.0cm] {$M_1$};

\node(1) [ above right of=T1, node distance = 2.0cm] {};
\node(dots) [right of= 1, node distance = 1.0cm] {$\dots$}; 
\node(2) [right of = dots, node distance = 1.0cm] {};
\node(3) [above of = dots, node distance = 1.4cm] {$\dots$};
\node(4) [below of = dots, node distance = 1.4cm] {$\dots$};

\node(Tk-1) [below right of=2, node distance = 2.0cm] {$T_{k-1}$};
\node(Sk) [above right of=Tk-1, node distance = 2.0cm] {$S_k$};
\node(lL') [below right of=Sk, node distance = 2.0cm] {$\lambda_p(M')$};
\node(em') [above of = lL', node distance = 2.82cm] {};
\node(Lk-1) [above left of=Sk, node distance = 2.0cm] {$M_{k-1}$};
\node(L0') [left of=em', node distance = 0.4cm] {$M_k$};
\node(L') [ right of = em', node distance = 0.4cm] {$M'_{}$};

\draw[-] (lL) to node {} (S1);
\draw[-] (S1) to node {} (T1);
\draw[-] (L0) to node {} (S1);
\draw[-] (L1) to node {} (S1);
\draw[-] (Lk-1) to node {} (Sk);
\draw[-] (Sk) to node {} (Tk-1);
\draw[-] (L0') to node {} (Sk);
\draw[-] (Sk) to node {} (lL');

\draw[->, dashed] (L0) to node {$\lambda_p$} (lL);
\draw[->, dashed] (L1) to node {$\lambda_p$} (T1);
\draw[->, dashed] (Lk-1) to node {$\lambda_p$} (Tk-1);
\draw[<-, dashed] (lL') to node {$\lambda_p$} (L0');

\draw[-] (T1) to node {} (1);
\draw[-] (Tk-1) to node {} (2);
\draw[-] (L1) to node {} (1);
\draw[-] (Lk-1) to node {} (2);
 
\draw[<-, dashed] (lL) to node {$\lambda_p$} (L);
\draw[->, dashed] (L') to node {$\lambda_p$} (lL');
\end{tikzpicture}
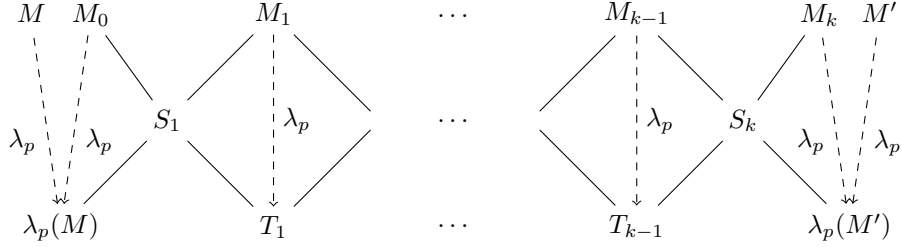
\begin{center}{\bf 3.3 Figure}
\end{center}
\end{center}

Then for any $i=0,1,\dots,k$, there are $\z$-lattices $M_i$ such that $M_0 \in \Psi_p(\lambda_p(M)) \cap \Phi_p(S_1)$, $M_k \in \Psi_p(\lambda_p(M')) \cap \Phi_p(S_k)$, and $M_j  \in \Psi_p(T_j) \cap \Phi_p(S_j) \cap \Phi_p(S_{j+1})$ for any $j=1,2,\dots,k-1$. Now by Lemma \ref{n=3}, there are $\z$-lattices $N_i$ such that
$$
\{\Gamma_{p,1}(N_i)^{\frac1p},\Gamma_{p,2}(N_i)^{\frac1p}\}=\{ M_{i-1},M_i\} \quad \text{and} \quad \lambda_p(N_i)=S_i
$$
for any $i=1,2,\dots,k$.  Since both $[M], [M_0]$ and $[M_k], [M']$ are connected by a path of length $4$  by Lemma \ref{nice}, $[M]$ and $[M']$ are connected by a path of length $k+8$.    \end{proof}


We investigate the graph $\mathfrak{G}_{L,p}(0)$ in more detail. Let $T \in \mathcal G_{L,p}(0)$ be a $\z$-lattice.
Note that the graph $Z(T,p)$ constructed in \cite{sp1} is slightly different from our graph (see also \cite{bh}). In fact, the graph $Z(T,p)$ is a tree having infinitely many vertices. However our graph is finite and might have a loop. Two vertices $[T_i], [T_j] \in \mathfrak{G}_{L,p}(0)$ are connected by an edge if and only if there are $\z$-lattices $T_i' \in [T_i]$ and $T_j' \in [T_j]$ such that $T_i'$ and $T_j'$ are connected by an edge in the graph $Z(T,p)$. If two lattices $T_i ,T_j \in \mathcal G_{L,p}(0)$ are spinor equivalent, then  both $[T_i]$ and $[T_j]$ are contained in the same connected component. Moreover, each connected component of $\mathfrak{G}_{L,p}(0)$ contains at most two spinor genera, and it contains
only one spinor genus if and only if $\bold{j}(p) \in P_D J_{\q}^T$, where $D$ is the set of positive rational numbers and 
$$
\bold{j}(p) = (j_{q}) \in J_\q \quad \text{such that $j_p =p$ and $j_q = 1$ for any prime $q \ne p$}.
$$
 We say that $\mathfrak{G}_{L,p}(0)$ is of $O$-type if each connected component of $\mathfrak{G}_{L,p}(0)$ contains only one spinor genus, and it is of $E$-type otherwise. If $\mathfrak{G}_{L,p}(0)$ is of $E$-type, then adjacent classes are contained in different spinor genera (for details, see \cite {bh}), that is, each connect component of the graph $\mathfrak{G}_{L,p}(0)$ is a bipartite graph.   

Assume that
\begin{equation}
\mathcal G_{L,p}(0)=\{[T_1],[T_2] \dots,[T_h]\} \quad \text{and} \quad \mathcal G_{L,p}(1)=\{[S_1], [S_2],\dots,[S_k]\}
\end{equation}
are {\it ordered} sets of equivalence classes in each genus. 
We define
$$
\mathfrak{M}= \left(\frac{r(T_i^p, S_j)}{o(T_i)}\right) \in M_{h,k}(\z) \  \text{and} \ \mathfrak N=\mathfrak N_{L,p}(0)=\left(\frac{r(T_i^p, S_j)}{o(S_j)}\right) \in M_{h,k}(\z).
$$
In fact, $\mathfrak{M}$ equals to $\mathfrak M_{L,p}(0)$, which is defined earlier. There is a nice relation between $\mathfrak M$, $\mathfrak N$ and  the {\it Eichler's Anzahlmatrix} $\pi_p(T)$ defined in \cite{e}. 

\begin{defn} \label{Eichler} Under the assumptions given above,
the matrix 
$$
\pi_p(T)=\begin{pmatrix}\displaystyle \frac{r(pT_i,T_j)}{o(T_i)}-\delta_{ij}\end{pmatrix} \quad (1\leq i,j \leq h) 
$$
is called the Eichler's Anzahlmatrix of $T$ at $p$.
\end{defn}

Note that $\pi_p(T)$ is independent of the choice of the lattice $T \in \mathcal G_{L,p}(0)$. 

\begin{lem} \label{omit} For any $\z$-lattices $T \in \mathcal G_{L,p}(0)$ and $S \in \mathcal G_{L,p}(1)$, we have $r(S^p, T) = r(T^p, S)$.
\end{lem}

\begin{proof} First we show that $\widetilde{R}(S^p,T)=R(S^p,T)$. Suppose that there is a $\sigma \in R(S^p,T)$ such that $T/\sigma(S^p) \simeq \z/p^2\z$. Then there is a basis for $T$ such that
$$
T=\z x_1+\z x_2+\z x_3 \quad \text{and} \quad \sigma(S^p)=\z x_1+\z x_2+\z (p^2x_3).
$$
Since $\mathfrak n(\sigma(S^p)) \subset p\z$, we have 
$$
Q(x_1) \equiv Q(x_2) \equiv 2B(x_1,x_2) \equiv 0 \pmod p.
$$
 This is a contradiction to the fact that $4dT$ is not divisible by $p$. Therefore the lemma follows from Lemma \ref{tilde r}. \end{proof}

For $\z$-lattices $X_1,X_2, Y_1$ and $Y_2$, we write $(X_1,X_2) \simeq (Y_1,Y_2)$ if $X_1 \simeq Y_1$ and  $X_2 \simeq Y_2$, or 
 $X_1 \simeq Y_2$ and  $X_2 \simeq Y_1$. 

\begin{prop}\label{Eichler-rel} Under the notations and assumptions given above, we have 
$$
\pi_p(T) + (p+1)I =\mathfrak{M}\cdot \mathfrak{N}^t.
$$
\end{prop}

\begin{proof} Let $\mathfrak U_{ij}$ be the set of sublattices $X$ of $T_j$ such that  
$$
X \simeq pT_i \quad \text{and}  \quad T_j/X \not \simeq \z/p\z\oplus \z/p\z\oplus \z/p\z, 
$$
and let $\mathfrak V_{ij}$ be the set of sublattices $Y$ of $T_j$ such that 
$$
Y^{\frac1p} \in \mathcal G_{L,p}(1) \quad \text{and} \quad \left(\Gamma_{p,1}(Y^{\frac 1p}),\Gamma_{p,2}(Y^{\frac1p})\right) \simeq  (T_i^p,T_j^p),
$$ 
where $\Gamma_{p,i}(Y^{\frac1p})$ is a sublattice of $Y^{\frac1p}$ with index $p$ defined in Lemma \ref{2-p}. Note that $\pi_p(T)_{ij}=\vert \mathfrak U_{ij}\vert$. Now we define a map $\Phi : \mathfrak U_{ij} \mapsto \mathfrak V_{ij}$ as follows. Assume that $X \in \mathfrak U_{ij}$. Then one may easily show that $T_j/X \simeq \z/p\z \oplus \z/p^2\z$. Hence there is a basis $x_1,x_2,x_3$ for $T_j$ such that 
$$
T_j=\z x_1+\z x_2+\z x_3 \quad \text{and}  \quad X=\z x_1+\z (px_2)+\z (p^2x_3).
$$
Since the integer $4d(T_j)$ is not divisible by $p$ and $Q(x_1) \equiv 0 \pmod {p^2}$, $2B(x_1,x_2) \equiv 0 \pmod p$, neither $Q(x_2)$ nor $2B(x_1,x_3)$ is divisible by $p$. 
Define $\Phi(X):=Y=\z x_1+\z(px_2)+\z (px_3)$. Clearly, $Y=\Lambda_p(T_j \cap \frac1p X)$. Hence it is independent of the choice of basis for $T_j$. Furthermore one may easily check that $\Phi(X)=Y \in \mathfrak V_{ij}$. 
Conversely, there are exactly two sublattices of $Y^{\frac 1p}$ with index $p$ whose norm is contained in $p\z$, and one of them is equal to $T_j^p$. If we define the other one, as a sublattice of $Y$, by $\Psi(Y)$, then $\Phi\circ \Psi=\Psi \circ \Phi=Id$.   
Therefore $\pi_p(T)_{ij}=\vert \mathfrak V_{ij}\vert$. Now from the definition,
$$
\vert \mathfrak V_{ij}\vert=\sum_{w=1}^k \displaystyle \frac {r(S_w^p,T_j)}{o(S_w)} \eta_w,
$$
where 
$$
\eta_w=\begin{cases}1  \qquad \text{if $(\Gamma_{p,1}(S_w),\Gamma_{p,2}(S_w))\simeq (T_j^p,T_i^p)$},\\
  0\qquad \text{otherwise.} \end{cases} 
$$
Since  $r(T_j^p,S_w)=r(S_w^p,T_j)$ by Lemma \ref{omit}, 
$$
\vert \mathfrak V_{ij}\vert=
\sum_{w=1}^k \displaystyle \frac {r(S_w^p,T_j)}{o(S_w)} \left(\frac{r(T_i^p,S_w)}{o(T_i)}-\delta_{ij}\right)=\begin{cases}
 \sum_{w=1}^k \mathfrak M_{iw} (\mathfrak N^t)_{wj}  \!&\text{if $i\ne j$,} \\
 \sum_{w=1}^k \mathfrak M_{iw} (\mathfrak N^t)_{wj} -(p+1) \! &\text{if $i=j$,} \\ \end{cases}
$$
by Lemma \ref{2-p}. The proposition follows from this.
\end{proof}

The following theorem states that the rank of $\mathfrak M_{L,p}(0)=\mathfrak M$ is related with some properties of the graph  $\mathfrak{G}_{L,p}(0)$. 
 
\begin{thm} \label{equicon} The followings are all equivalent: 
\begin{enumerate}
	\item $\mathfrak{G}_{L,p}(0)$ is of $O$-type;
	\item $\text{rank}({\mathfrak{M}}) = h$;
	\item $\pi_p(T)$ does not have an eigenvalue $-(p+1)$;
	\item $g^+(\mathcal G_{L,p}(0)) = g^+(\mathcal G_{L,p}(1))$. 
\end{enumerate}
Furthermore, if $\mathfrak{G}_{L,p}(0)$ is of $E$-type, then $g^+(\mathcal G_{L,p}(0)) = 2 g^+(\mathcal G_{L,p}(1))$, where $g^+(\mathcal G_{L,p}(0))$ is the number of spinor genera in $\mathcal G_{L,p}(0)$.
\end{thm}

\begin{proof} 
{\bf (1) $\Leftrightarrow$ (2): } Assume that $\mathfrak{G}_{L,p}(0)$ is of $O$-type. Without loss of generality, we may assume that $\mathfrak{G}_{L,p}(0)$ is connected, that is, every $\z$-lattice in  $\mathfrak{G}_{L,p}(0)$ is spinor equivalent. It is well known that the rank of an incidence matrix of a connected graph $G(V,E)$ over $\mathbb F_2$ is $\vert V\vert-1$. Furthermore if the graph $G$ contains an odd cycle, then the rank of the incidence matrix of $G$ over $\mathbb Q$ is equal to the number of vertices. Hence it suffices to show that the graph $\mathfrak{G}_{L,p}(0)$ contains an odd cycle, even though it might contains a loop. Assume that $[T_1]$ and $[T_2]$ be adjacent vertices in $\mathfrak{G}_{L,p}(0)$. Since they are spinor equivalent, there is an isometry $\sigma \in O(V)$ and $\Sigma=(\Sigma_p) \in J_V'$ such that $T_1=\sigma\Sigma(T_2)$, where $V=\q \otimes T_1$. 
Let $\Phi=\{ q \in  P-\{p\} \mid (\sigma^{-1}(T_1))_q=(T_2)_q\}$ and $\Psi=P-(\Phi\cup \{p\})$, where $P$ is the set of all primes. Now by Strong Approximation Theorem for Rotations, for any $\epsilon >0$, there is a rotation $\tau \in O'(V)$ such that
$$ 
\| \tau-\Sigma_q \|_q <\epsilon \ \ \text{for any $q \in \Psi$} \quad \text{and} \quad \| \tau \|_q=1 \ \ \text{for any $q \in \Phi$}.
$$
Therefore we have 
$$
\sigma^{-1}(T_1)_q=\tau(T_2)_q \ \  \text{for any $q \ne p$} \quad \text{and} \quad  \Sigma_p\circ \tau^{-1} (\tau(T_2)_p)=\sigma^{-1}(T_1)_p,
$$ 
where $\Sigma_p\circ \tau^{-1} \in O'(V_p)$. Consequently, there is an even integer $n$ and a basis $\{ x_1,x_2,x_3\}$ for $\tau(T_2)$ such that
$$
\tau(T_2)= \z x_1+\z x_2+\z x_3 \quad \text{and} \quad \sigma^{-1}(T_1) = \z(p^nx_1)+\z(p^{-n}x_2)+\z x_3,
$$
by Lemma 4.2 of \cite {bh}. This implies that there is a path from $[T_1]$ to $[T_2]$ with even edges, and hence the graph  $\mathfrak{G}_{L,p}(0)$ contains an odd cycle.  

Assume that $\mathfrak{G}_{L,p}(0)$ is of $E$-type. Since any two adjacent vertices are contained in different spinor genera in this case, it is a bipartite (multi-) graph. Therefore the rank of the matrix $\mathfrak M_{L,p}(0)$ is $h-1$.

\noindent {\bf (2) $\Leftrightarrow$ (3) :}  Note that $\text{rank}(\mathfrak M)=\text{rank}(\mathfrak M\mathfrak N^t )$. Hence the assertion follows directly from Proposition \ref{Eichler-rel}.

\noindent {\bf (1) $\Leftrightarrow$ (4) :}  Note that $g^+(\mathcal L) = [J_\q : P_D J_\q^{\mathcal L}]$ for any genus $\mathcal L$ with rank greater than $2$. Since 
$$
P_D J_\q^{\mathcal G_{L,p}(1)} = P_D J_\q^{\mathcal G_{L,p}(0)} \cup \bold{j}(p) \cdot P_D J_\q^{\mathcal G_{L,p}(0)},
$$ 
$g^{+}(\mathcal G_{L,p}(1))=g^{+}(\mathcal G_{L,p}(0))$ if and only if $\bold{j}(p) \in P_D J_\q^{\mathcal G_{L,p}(0)}$, that is, $\mathfrak{G}_{L,p}(0)$ is of $O$-type. Furthermore if $\mathfrak{G}_{L,p}(0)$ is of $E$-type, then $g^+(\mathcal G_{L,p}(0)) = 2g^+(\mathcal G_{L,p}(1))$.
\end{proof}



Now, we consider the general case. 
For any positive integer $m$, we say that a graph $\mathfrak G_{L,p}(m)$ is of $E$-type if $m$ is even and 
$\mathfrak G_{L,p}(0)$ is of $E$-type, and $O$-type otherwise.  

Assume that  $\mathfrak G_{L,p}(m)$ is of $E$-type and $M \in \mathcal G_{L,p}(m)$. Since the map $\lambda_p^{\frac m2} : \spn(K) \to \spn(\lambda_p^{\frac m2}(K))$ is surjective for any $K \in \mathcal G_{L,p}(m)$, 
there is a $\z$-lattice $M' \in \mathcal G_{L,p}(m)$ such that $M' \not \in \spn(M)$ and $[M']$ is connected to $[M]$ by a path by Lemma \ref{connect}. Furthermore, since $g^{+}(\mathcal G_{L,p}(m))=g^{+}(\mathcal G_{L,p}(0))$ for any even $m$, every $\z$-lattice $M'$ satisfying the above condition forms a single spinor genus.  
 From the existence of such a $\z$-lattice $[M']$, we may define  
$$
\text{Cspn}(M)=\begin{cases} \text{spn}(M) \quad &\text{if $\mathfrak G_{L,p}(m)$ is of $O$-type,}\\
 \text{spn}(M) \cup \spn(M')  \quad &\text{otherwise},
 \end{cases}
$$

\begin{lem} \label{coco} For a $\z$-lattice $M \in \mathcal G_{L,p}(m)$, the set of all vertices in the connected component of $\mathfrak G_{L,p}(m)$ containing $[M]$ is the set of equivalence classes in $\text{Cspn}(M)$.
\end{lem}

\begin{proof} First, we prove the case when $m=1$. Assume that $M' \in \text{spn}(M)$. Then there are $\sigma \in P_V$ and $\Sigma \in J'_V$ such that $M'=\sigma\Sigma M$ (see \cite{om2}). Since $\Gamma_{p,i}(M)$'s are the only sublattices of $M$ with index $p$ whose norm is $p\z$, we have 
$$
\{\sigma\Sigma(\Gamma_{p,1}(M)^{\frac1p}),\sigma\Sigma(\Gamma_{p,2}(M)^{\frac1p})\}=\{\Gamma_{p,1}(M')^{\frac1p},\Gamma_{p,2}(M')^{\frac1p}\}.
$$
Hence $\Gamma_{p,1}(M)^{\frac1p} \in \spn(\Gamma_{p,1}(M')^{\frac1p}) \cup \spn(\Gamma_{p,2}(M')^{\frac1p})$. Therefore by Lemma \ref{n=2}, $[M']$ and $[M]$ are connected by a path in $\mathfrak G_{L,p}(1)$. Furthermore, as edges of the graph 
 $\mathfrak G_{L,p}(0)$, $[M]$ and $[M']$ are contained in the same connected component. Since the number of connected components in $\mathfrak G_{L,p}(0)$ equals to $g^{+}(\mathcal G_{L,p}(1))$ by Theorem \ref{equicon}, each spinor genus in $\mathcal G_{L,p}(1)$ forms a connected component in $\mathfrak G_{L,p}(1)$. Furthermore, since $g^{+}(\mathcal G_{L,p}(2m+1))=g^{+}(\mathcal G_{L,p}(1))$, $\spn(\lambda_p^{\frac m2}(M))=\spn(\lambda_p^{\frac m2}(M'))$ if and only if $\spn(M)=\spn(M')$ for any $M,M' \in \mathcal G_{L,p}(2m+1)$. Therefore by Lemma \ref{connect},   the set of all vertices in the connected component of $\mathfrak G_{L,p}(m)$ containing $[M]$ is the set of equivalence classes in $\text{Cspn}(M)$
 for any odd $m$. The proof of even case is quite similar to this.   
\end{proof}

\begin{thm} For any non-negative integer $m$, the graph $\gr$ has an odd cycle (including a loop) if and only if $\gr$ is of $O$-type.
\end{thm}
 
\begin{proof} We already proved the case when $m=0$ in Theorem \ref{equicon}.  Assume that $m=1$. Let $T \in \mathcal G_{L,p}(0)$ be any $\z$-lattice. Then there are at least three $\z$-lattices, say $S_1, S_2, S_3$, in $\Phi_p(T) \cap \mathcal G_{L,p}(1)$. Now by Lemma \ref{n=2}, $[S_i]$ and $[S_j]$ are connected by an edge for any $1 \le i\ne j \le 3$. Hence the graph $\mathfrak G_{L,p}(1)$ contains a cycle of length $3$ or a loop. For the general case, we may apply Lemma \ref{connect} to prove the theorem.    
\end{proof}

\section{Representations of integers by ternary quadratic forms}

Throughout this section, we assume that a $\z$-lattice $L$ and a prime $p$ satisfies all conditions given in Section 3. 
For a nonnegative integer $m$, let $T \in \mathcal G_{L,p}(m)$ be a ternary $\z$-lattice and let $S \in \mathcal G_{L,p}(m+1)$ be a ternary $\z$-lattice such that $r(T^p,S) \ne 0$. This implies that $[T]$ is one of vertices   contained in the edge $[S]$ in the graph $\mathfrak G_{L,p}(m)$.  
We assume that    
 \begin{equation}
\text{Cspn}(T)=\{[T_1],[T_2] \dots,[T_u]\} \quad \text{and} \quad \text{Cspn}(S)=\{[S_1],[S_2], \dots,[S_v]\}
\end{equation}
are {\it ordered} sets of equivalence classes. 
The aim of this section is to show that 
if $m \le 2$, then there are rational numbers $a_i$ and $b_i$ such that for any integer $n$ (any integer $n$ divisible by $p$ only when $m=2$),
\begin{equation}
r(n,T)=\sum_{i=1}^v \left(a_{i}r(pn,S_i)+b_{i}r(p^3n,S_i)\right)+(\text{some extra term}).
\end{equation} 
 For a while, we assume that $m$ is an arbitrary nonnegative integer. The following two propositions will be used repeatedly.

\begin{prop}\label{prop1}
For any integer $n$, 
$$
\frac{r(pn, S)}{o(S)}=\sum_{i=1}^{u} \frac{r(T_i^p,S)}{o(S)} \frac {r(n,T_i)}{o(T_i)} - \frac{r(pn,\Lambda_p(S))}{o(S)}.
$$
\end{prop}
\begin{proof}
 By Weak Approximation Theorem, there exists a basis $\{x_1, x_2, x_3 \}$ for $S$ such that
$$
(B(x_i,x_j))\equiv\begin{pmatrix}0&\frac12\\ \frac12&0\end{pmatrix}\perp \langle p^{m+1} \delta\rangle \ (\text{mod} \ p^{m+2}),
$$
where $\delta$ is an integer not divisible by $p$. As in Lemma \ref{2-p}, let 
$$
\aligned
&\Gamma_{p,1}(S) = \z px_1 + \z x_2+ \z x_3, &\Gamma_{p,2}(S) = \z x_1 + \z px_2+ \z x_3. 
\endaligned
$$
 Since $Q(x) \equiv a_1a_2\ (\text{mod}\ p)$ for any $x=a_1x_1+a_2x_2+a_3x_3 \in S$, we have  $Q(x)\equiv0 \ (\text{mod}\ p)$ if and only if $a_1\equiv0 \ (\text{mod}\ p)$ or $a_2\equiv0 \ (\text{mod} \ p)$. Hence
$$
x \in R(pn,S) \quad \text{if and only if} \quad x \in R(pn,\Gamma_{p,1}(S)) \cup R(pn,\Gamma_{p,2}(S))
$$
Furthermore since $\Gamma_{p,1}(S) \cap \Gamma_{p,2}(S)=\Lambda_p(S)$, we have
$$
r(pn,S)=r(pn, \Gamma_{p,1}(S))+r(pn,\Gamma_{p,2}(S))-r(pn,\Lambda_p(S))
$$ 
for any integer $n$. Note that $\Gamma_{p,1}(S)$ and $\Gamma_{p,2}(S) \in \text{gen}(T^p)$ are the only sublattices of $S$ that are contained in $\gen(T^p)$. Furthermore, since the edge $[S]$ in $\mathfrak G_{L,p}(0)$ contains the vertex $[T]$ by assumption, we have  
$\Gamma_{p,1}(S)^{\frac1p}, \Gamma_{p,2}(S)^{\frac1p} \in \text{Cspn}(T)$. Now for any $\z$-lattice $T_i \in \text{Cspn}(T)$, the number of sublattices in $S$ that are isometric to $T_i^{p}$ is $\frac{r(T_i^{p},S)}{o(T_i)}$. The proposition follows from this.   \end{proof}

\begin{prop} \label{prop2}
For any integer $n$,
$$
\frac {r(pn,T)}{o(T)}=\begin{cases} \displaystyle\sum_{j=1}^{v} \frac{r(S_j^p,T)}{o(T)} \frac{r(n,S_j)}{o(S_j)} -p\cdot \frac{r(n,T^p)}{o(T)} \quad &\text{if $m=0$,} \\
                     \displaystyle \sum_{j=1}^{v} \frac{\tilde{r}(S_j^p,T)}{o(T)}\frac{r(n,S_j)}{o(S_j)} +\frac{r(pn,\Lambda_p(T))}{o(T)}-2p\cdot \frac{r(n,T^p)}{o(T)} \quad &\text{otherwise.} \\  \end{cases}    
$$
\end{prop}
\begin{proof}

If we take $\epsilon = 0$ and $L=T$ in Lemma \ref{2}, then we have 
$$
r(pn,T)=\sum_{M\in\Omega_p(0, T)}r(pn,M)-(s_p(0,T)-1)r(n,T^p).
$$
First, assume that $m=0$.  Let $M\in \Omega_p(0,T)$ be a $\z$-lattice. Then by Lemmas \ref{local} and \ref{local-2}, 
$$
M_p\simeq \begin{pmatrix} 0&\frac p2\\ \frac p2&0 \end{pmatrix}\perp  \langle -4p^2dT\rangle \quad \text{and}  \quad  M_q\simeq T_q ~(q \ne p).
$$ 
Hence $M\in\gen(S^p)$. Furthermore, since $r(T^p,M^{\frac1p})=\tilde {r}(M,T)\ne 0$ and $r(T^p,S)=\tilde{r}(S^p,T) \ne 0$ by 
Lemma \ref{tilde r}, $M^{\frac1p} \in \text{Cspn}(S)$ by Lemmas \ref{n=2} and \ref{coco}.  
Conversely, if $M^{\frac1p} \in \text{Cspn}(S)$ satisfies $\tilde{r}(M,T) \ne 0$, then $M$ is isometric to a $\z$-lattice in $\Omega_p(0,T)$. Note that the number of lattices in $\Omega_p(0,T)$ that are isometric to $S^p$ is $\frac {r(S^p,T)}{o(S)}$ and  $s_p(0,T)=p+1$. The proof of the case when $m \ge 1$ is quite similar to this, except that there is a unique $\z$-lattice in $\Omega_p(0,T)$ that is not contained in $\gen(S^p)$, which is, in fact, $\Lambda_p(T)$, and $s_p(0,T)=2p+1$.    
\end{proof}

We define
$$
\mathcal M_{L,p}(m)= \left(\frac{r(T_i^p, S_j)}{o(T_i)}\right) \in M_{u,v}(\z) \ \ \text{and} \ \ \mathcal N_{L,p}(m)=\left(\frac{r(T_i^p, S_j)}{o(S_j)}\right) \in M_{u,v}(\z).
$$
Note that these two matrices depend on the order of each set $\text{Cspn}(\cdot)$, and $\mathcal M_{L,p}(0)$ is one of block diagonal components of $\mathfrak M_{L,p}(0)$ if we take a suitable order in (3.1).
 For any integer $n$, we define  vectors
$$
\mathbf R(n,\text{Cspn}(T)) =\displaystyle\left(\frac{r(n,T_1)}{o(T_1)},\displaystyle \frac{r(n,T_2)}{o(T_2)},\dots,\displaystyle\frac{r(n,T_u)}{o(T_u)}\right)^t,
$$
 $$
  \mathbf{R}^{\sharp}(n, \text{Cspn}( \lambda_p^m(T))) = 
  \displaystyle \left(\frac{r(n,\lambda_p^m(T_1))}{o(T_1)},\frac{r(n,\lambda_p^m(T_2))}{o(T_2)},\dots,\frac{r(n,\lambda_p^m(T_u))}{o(T_u)}\right)^t.
$$
 Similarly, we define $\mathbf R(n,\text{Cspn}(S))$ and  $\mathbf{R}^{\sharp}(n, \text{Cspn}( \lambda_p^m(S)))$. 
 If $\text{Cspn}(M)=\spn(M)$, then we use $\mathbf R(n,\spn(M))$ rather than $\mathbf R(n,\text{Cspn}(M))$. 
\begin{thm} \label{rel0O} Let $T$ and $S$ be ternary $\z$-lattices satisfying all conditions given above when $m=0$. If the graph $\mathfrak G_{L,p}(0)$ is of $O$-type, then
we have
$$
p \mathbf R(n,\text{spn}(T^p))=\mathcal{M}\cdot \mathbf R(n,\text{spn}(S))-(\mathcal{M}\cdot \mathcal{N}^t)^{-1}\mathcal{M}\cdot (\mathbf R(p^2n,\text{spn}(S))+\mathbf R(n,\text{spn}(S))).
$$
\end{thm}

\begin{proof} By Lemma \ref{omit} and Propositions \ref{prop1}, \ref{prop2}, we have the following two equalities:
\begin{eqnarray} 
\mathbf R(pn,\text{spn}(S))=\mathcal {N}^t \cdot\mathbf R(n,\text{spn}(T))-\mathbf R^{\sharp}(pn, \text{spn}(\Lambda_p(S))),  \\
\mathbf R(pn,\text{spn}(T))=\mathcal{M}\cdot \mathbf R(n,\text{spn}(S))-p\mathbf R(n,\text{spn}(T^p)).
\end{eqnarray}
Since $\lambda_p(\lambda_p(S_i)) \simeq S_i$ for any $S_i \in \text{spn}(S)$, we have 
$$
\mathbf R^{\sharp}(p^2n,\text{spn}(\Lambda_p(S)))=\mathbf R(n,\text{spn}(S)).
$$ 
Hence 
\begin{equation}
\mathbf R(p^2n,\text{spn}(S))=\mathcal N^t\cdot \mathbf R(pn,\text{spn}(T))-\mathbf R(n,\text{spn}(S)).
\end{equation}
Note that 
$$
\mathbf O(\text{spn}(T))\cdot \mathcal N=\mathcal M \cdot \mathbf O(\text{spn}(S)), 
$$
where $\mathbf O(\text{spn}(T))$ is the $u \times u$ diagonal matrix with entries $o(T_i)^{-1}$.
 Furthermore, since we are assuming that $\text{rank}(\mathcal{M}) = u$, the $u\times u$ square matrix $\mathcal M\cdot \mathcal N^t$ is invertible. Therefore the equation follows directly from (4.4) and (4.5).
\end{proof}

Now assume that $\mathfrak{G}_{L,p}(0)$ is of $E$-type, then $\text{Cspn}(T)$ consists of two spinor genera and each connected component is a bipartite graph. Hence the rank of the matrix $\mathcal{M}$ is $u-1$ and $\mathcal{M}\cdot \mathcal{N}^t$ is no longer invertible. To get a similar result for an $E$-type graph, we need to make some adjustments.
		
Assume that $\text{Cspn}(T)=\spn(T) \cup \spn(\tilde{T})$ and 
$$
\spn(T)= \{[ T_{i_1}] , \ldots , [T_{i_{a}}] \}, \quad  \spn(\tilde{T}) = \{ [T_{j_1}] , \ldots, [T_{j_{b}}] \},
$$
where $\{i_1,i_2,\dots,i_a,j_1,\dots, j_b\}=\{1,2,\dots,u\}$.	Note that
$$
 w(\spn(T'))=\sum_{[K]\in \spn(T')} \frac{1}{o(K)},
 $$
is independent of $T'$ for any $T' \in \gen(T)$. Define 
$$\epsilon_{l} = \begin{cases}
	 w(\spn(T))^{-1} \quad &\text{if } l \in \{ i_1 , \ldots , i_a\}, \\ -w(\spn(T))^{-1} \quad &\text{if } l \in  \{ j_1 , \ldots , j_b \}, \end{cases} 
$$
 and define a $u \times (v+1)$ matrix $ \mathcal{\tilde{N}}= (n_{ij})$ by
$$ 
n_{ij} =\begin{cases}
	\displaystyle \frac{r(T_i ^p , S_j )}{o(S_j)}\quad &\text{if } j \le v, \\
	\epsilon_i \quad &\text{if } j= v+1. \end{cases}  
$$ 	
\begin{lem}\label{rankN}
The rank of the matrix $\mathcal{\tilde{N}}$ defined above is $u$. 
\end{lem}

\begin{proof}
Let $\mathbf n_i$ be the $i$-th row vector of the matrix $ \mathcal{\tilde{N}}$. Suppose that $\alpha_1 \mathbf n_1 + \cdots +\alpha_u \mathbf n_u= 0$ for some integers $\alpha_i$, that is, 
\begin{equation} 
\begin{cases}
&\alpha_1 \displaystyle \frac{r(T_1 ^p, S_j)}{o(S_j)} + \cdots +\alpha_u \displaystyle \frac{r(T_u ^p, S_j)}{o(S_j)} = 0 \ \  \text{for any} \ j=1,\dots,v, \\
&\alpha_1 \epsilon_1 + \cdots + \alpha_u \epsilon_u = 0.  
\end{cases}
\end{equation}
For any $j$ such that $1\le j \le v$, the edge named by $[S_j]$ contains two vertices, one of them, say $[T_{i_e}]$, is contained in $\spn(T)$ and 
the other, say $[T_{j_f}]$, is contained in $\spn(\tilde{T})$. Hence the first equation in $(4.6)$ implies that 
$$
 \alpha_{i_e} \frac{r(T_{i_e} ^p, S_j)}{o(S_j)} + \alpha_{j_f} \frac{r(T_{j_f} ^p, S_j)}{o(S_j)} = 0. 
$$
Therefore $\alpha_{i_e}\cdot\alpha_{j_f} \le 0$. Since the subgraph of $\mathfrak G_{L,p}(0)$ consisting of vertices in $\text{Cspn}(T)$ is a connected bipartite graph, each $\alpha_{i_e}$ ($\alpha_{j_f}$) is 0, or it has the same sign to $\alpha_{i_1}$ ($\alpha_{j_1}$, respectively). Therefore $\alpha_ l =0$ for any $l = 1,\ldots, u$ and $\text{rank}( \mathcal{\tilde{N}}) = u$. This completes the proof.
\end{proof}

For a vector $\mathbf v=(v_1,\dots,v_n)$, we define $(\mathbf v,w_1,\dots,w_s)=(v_1,\dots,v_n,w_1,\dots,w_s)$. 
Note that the equation $(4.5)$ implies that
\begin{equation}
\widetilde{\mathbf R}:={\mathcal{\tilde{N}}}^t \cdot \mathbf{R}(pn, \text{Cspn}(T)) = \left( \begin{array}{c} \\ \mathbf R(p^2n,\text{spn}(S))+\mathbf R(n,\text{spn}(S)) \\  \\ r(pn, \spn(T)) - r(pn,\spn(\tilde{T})) \end{array}\right),
\end{equation}
where 
$$ r(pn, \spn(T)) =\frac 1{w(\spn(T))} \cdot \sum_{[T_i] \in \spn(T)} \frac{r(pn,T_i)}{o(T_i)}. $$

\begin{thm} \label{rel0E} If  $\mathfrak G_{L,p}(0)$ is of $E$-type, then we have
$$p\mathbf{R}(n, \text{Cspn}(T^p)) = \mathcal{M}\cdot \mathbf{R}(n,\text{spn}(S))-(\tilde{\mathcal{N}}\cdot \tilde{\mathcal{N}}^t )^{-1}\tilde{\mathcal{N}}\cdot \widetilde{\mathbf{R}}.
$$
\end{thm}

\begin{proof}
From the above lemma, we know that $\text{rank}(\tilde{\mathcal{N}})=u$. The theorem follows directly from the equations (4.4) and (4.7).
\end{proof}

Note that $ r(pn, \spn(T)) - r(pn,\spn(\tilde{T})) $ can easily be computed by the formula given in \cite{sp3}. 
\begin{exam}
Let  $p=11$ and $L=\langle1, 1, 16\rangle$. Then
$$
\aligned
&\mathcal G_{L,p}(0)/\sim =\Bigg\{ T_1=\begin{pmatrix}1&0&0\\0&1&0\\0&0&16\end{pmatrix},~ T_2= \begin{pmatrix} 2&0&-1\\0&2&1\\-1&1&5\end{pmatrix}\Bigg\},\\
&\mathcal G_{L,p}(1)/\sim =\Bigg\{ S_1=\begin{pmatrix}3&1&1\\1&6&-1\\1&-1&11\end{pmatrix},~ S_2 =\begin{pmatrix} 6&2&3\\2&6&1\\3&1&7\end{pmatrix}\Bigg\}.
\endaligned
$$ 
 One may easily compute that $\mathcal{M}=\begin{pmatrix}1&1\\1&1\end{pmatrix}$ and $ \mathcal{N}=\begin{pmatrix}8&4\\8&4\end{pmatrix}$.  Since $\text{rank}(\mathcal{M})=1$, the graph $\mathfrak G_{L,p}(0)$ is of $E$-type by Theorem  \ref{equicon}. Note that $\tilde{\mathcal{N}} = \begin{pmatrix}8&4&16\\ 8&4&-16\end{pmatrix}$.
Therefore, by Theorem \ref{rel0E}, we have
$$
\aligned
11r(n, T_1^{11}) &=\frac{38}{5}r(n, S_1) - \frac{2}{5}r(11^2n, S_1) + \frac{39}{10}r(n, S_2) - \frac{1}{10}r(11^2n, S_2) \\&- \left(\frac{1}{2}r(11n, T_1) - \frac{1}{2}r(11n, T_2)\right), \\
11r(n, T_2^{11}) &=\frac{38}{5}r(n, S_1) - \frac{2}{5}r(11^2n, S_1) + \frac{39}{10}r(n, S_2) - \frac{1}{10}r(11^2n, S_2) \\&+ \left(\frac{1}{2}r(11n, T_1) - \frac{1}{2}r(11n, T_2)\right).
\endaligned
$$

Note that by  Korollar 2 of \cite{sp3}, one may easily check that 
$$
r(11n, T_1) - r(11n, T_2) = \begin{cases} 0 \quad &\text{if $n \neq 11m^2$,} \\ \displaystyle \left(\frac{1-(-1)^m}{2}\right)\cdot(-1)^{\frac{m+1}{2}}\cdot44m \quad &\text{if $n = 11m^2$.}  \end{cases} 
$$
\end{exam}


\begin{thm}
Let  $T \in \mathcal G_{L,p}(1)$  and $S \in \mathcal G_{L,p}(2)$ be ternary $\z$-lattices satisfying $r(T^p,S) \ne 0$. Then we have
$$ 
\aligned & (3p^2 -p) \cdot r(n,T)=\sum_{[\tilde{S}] \in \gen(S)} \frac{\tilde{r}(\tilde{S}^p,T)}{o(\tilde{S})} \left(\frac{3p}2~r(pn, \tilde{S}) - \frac p{p-1}~r(p^3n, \tilde{S})\right) \\ &+ \frac 1{p-1} \left( o(\Gamma_{p,1}(T)) \sum_{\substack{[\tilde{S}] \in \gen(S) \\ \lambda_p(\tilde{S}) \simeq \Gamma_{p,1}(T)^{\frac1p}}} \frac{r(p^3 n,\tilde{S})}{o(\tilde{S})} + o(\Gamma_{p,2}(T)) \sum_{\substack{[\tilde{S}] \in \gen(S) \\ \lambda_p(\tilde{S}) \simeq\Gamma_{p,2}(T)^{\frac1p} }} \frac{r(p^3 n,\tilde{S})}{o(\tilde{S})}\right).
\endaligned
$$
\end{thm}

\begin{proof} First, we assume that 
$$
\Phi_p(\lambda_p(S))=\{T=T_1,T_2,\dots,T_{p+1}\} \quad \text{and} \quad 
\Psi_p(\lambda_p(S))=\{S=S_1,S_2,\dots,S_{\frac{p(p+1)}2}\}.
$$
Without loss of generality, we may assume that $\lambda_p(S)=\Gamma_{p,1}(T)^{\frac1p}$.
Define, for any integer $n$, 
$$
\mathbf{R}(n,\Phi_p(\lambda_p(S)))=(r(n,T_1),r(n,T_2),\dots,r(n,T_{p+1}))^t
$$
and
$$
\mathbf{R}(n,\Psi_p(\lambda_p(S)))=\left(r(n,S_1),r(n,S_2),\dots,r\left(n,S_{\frac{p(p+1)}2}\right)\right)^t.
$$

We also define a vector $\mathbf{I}(n,\lambda_p(S))=r(n,\lambda_p(S))\cdot (1,1,\dots,1)^t$ of length $\frac{p(p+1)}2$.  
Now by Proposition \ref{prop1}, we have
$$
\mathbf{R}(pn,\Psi_p(\lambda_p(S)))=U \cdot \mathbf{R}(n,\Phi_p(\lambda_p(S)))-\mathbf{I}\left(\frac np,\lambda_p(S)\right),
$$
 where $U^t \in M_{(p+1)\times \frac{p(p+1)}2}(\z)$ is the incidence matrix of the complete graph of order $p+1$ by Lemma \ref{n=2}.
Therefore $U^tU=(p-1)I+J$ and 
$$
((U^tU)^{-1}U^t)_{ij}=\begin{cases} \displaystyle \frac1p \quad &\text{if $r(T_i^p,S_j)\ne 0$},\\
                          
                                            \displaystyle  \frac {-1}{p(p-1)}  \quad &\text{if $r(T_i^p,S_j)=0$}. \end{cases}
$$
Here $J$ is a matrix of ones.  Therefore we have 
\begin{equation}
r(n,T)=\frac1p\sum_{\mathbf 1} r(pn,S) -\frac{1}{p(p-1)}\sum_{\mathbf 2} r(pn,S)+\frac12 r\left(\frac np,\lambda_p(S)\right),
\end{equation} 
where $\sum_{\mathbf 1}$ is the summation of all lattices $S'$ in $\Psi_p(\lambda_p(S))$ such that $r(T^p, S')\ne 0$ and                                   
$\sum_{\mathbf 2}$ is the summation of all lattices $S'$ in $\Psi_p(\lambda_p(S))$ such that $r(T^p, S')=0$. We define, for simplicity, $U_1(pn,S)=\sum_{\mathbf 1} r(pn,S)$ and  $U_2(pn,S)=\sum_{\mathbf 2} r(pn,S)$. Now, by Proposition \ref{genTL},
we have
\begin{equation}
\begin{array} {rl}
p\cdot r(pn,\lambda_p(S))+ \displaystyle \frac{p(p-1)}2 r\left( \displaystyle \frac np,\lambda_p(S)\right)=&o(\lambda_p(S))r(pn,\gen_p^{\lambda_p(S)}(S))\\
 =&\displaystyle \sum_{i=1}^{\frac{p(p+1)}2}r(pn,S_i)\\
 =&U_1(pn,S)+U_2(pn,S).\end{array}
\end{equation}

Let $\widetilde{S}$ be a $\z$-lattice such that $\lambda_p(\widetilde{S})=\Gamma_{p,2}(T)^{\frac1p}$. We may  similarly define 
$\mathbf{R}(n,\Psi_p(\lambda_p(\widetilde{S})))$, $U_1(pn,\widetilde{S})$ and $U_2(pn,\widetilde{S})$.  Then, equations (4.8) and (4.9) hold even if we replace $S$ by $\widetilde{S}$. Furthermore, by Proposition \ref{prop2}, 
\begin{equation}
\begin{array} {rl}
r(p^2n,T)+(2p-1)r(n,T)=& \displaystyle \sum_{[S'] \in \gen(S)} \frac{\tilde{r}((S')^p,T)}{o(S')} r(pn,S')\\
                      =&U_1(pn,S)+U_1(pn,\widetilde{S}).\end{array}
\end{equation}    
By combining (4.8)$\sim$(4.10), we have 
$$
\begin{array} {rl}
\frac{3p^2-p}2r(n,T)=&\!\!\! p(U_1(pn,S)+U_1(pn,\widetilde{S}))-p\left(\frac1p U_1(p^3n,S)-\frac1{p(p-1)}U_2(p^3n,S)\right)\\
                -&\!\!\!\frac{p(p-1)}2\left(\frac1p
                     U_1(pn,S)-\frac1{p(p-1)}U_2(pn,S)\right)\!\!-\!\!\frac12\left(U_1(pn,S)+U_2(pn,S)\right)\\
                      =&\!\!\!\displaystyle \frac p2U_1(pn,S)+pU_1(pn,\widetilde{S})-\left(U_1(p^3n,S)-\frac1{p-1}U_2\left(p^3n,S\right)\right). 
\end{array}
$$
Since the above equation holds even if we exchange $S$ for $\widetilde{S}$, we have
$$
\begin{array} {rl}
(3p^2-p)r(n,T)=\!\!\!&\displaystyle \frac {3p}2\left(U_1(pn,S)+U_1(pn,\widetilde{S})\right)-\frac p{p-1}\left(U_1(p^3n,S)+U_1(p^3n,\widetilde{S})\right)\\
+\!\!\!&\displaystyle \frac1{p-1}\left(U_1(p^3n,S)+U_2(p^3n,S)+U_1(p^3n,\widetilde{S})+U_2(p^3n,\widetilde{S})\right).
\end{array}
$$
This completes the proof.                                          
\end{proof}

\begin{rmk} {\rm In the above theorem, one may easily check that the sets $\Psi_p(\lambda_p(S))$ and $\Psi_p(\lambda_p(\widetilde{S}))$ are contained in $\text{Cspn}(S)$.} 
\end{rmk}

Assume that $m=2$. Recall that $T \in \mathcal G_{L,p}(2)$  and $S \in \mathcal G_{L,p}(3)$ are ternary $\z$-lattices satisfying  $r(T^p,S) \ne 0$. 
If we define $\epsilon_l$ and $\tilde{\mathcal{N}}$ as before for the $E$-type, then Lemma \ref{rankN} still holds under this situation.


\begin{thm} \label{rel2} Let $T$ and $S$ be ternary $\z$-lattices satisfying all conditions given above. Assume that the graph $\mathfrak G_{L,p}(2)$ is of $O$-type. If $n$ is not divisible by $p$, then
we have 
\begin{equation}
\mathbf{R}(n,\text{spn}(T)) =(\mathcal{N}\cdot \mathcal{N}^t)^{-1} \mathcal{N}\cdot \mathbf{R}(pn,\spn(S)). \\\\
\end{equation}
If $n$ is divisible by $p$, then $\mathbf{R}(n,\text{spn}(T))$ is equal to  
$$
\frac1{2p-1}\left(\mathcal{M}\cdot \mathbf R(pn,\text{spn}(S))-(\mathcal{N}\cdot \mathcal{N}^t)^{-1}\mathcal{N}\cdot (\mathbf R(p n,\text{spn}(S))+\mathbf R(p^3 n,\text{spn}(S)))\right).
$$
If  $\mathfrak G_{L,p}(2)$ is of $E$-type, then we have
$$
\mathbf{R}(n, \text{Cspn}(T)) =\begin{cases} (\tilde{\mathcal{N}}\cdot \tilde{ \mathcal{N}}^t)^{-1} \tilde{\mathcal{N}}\cdot \widetilde{\mathbf{R}}_1 & \text{if } p \nmid n,\\\\ \displaystyle \frac 1 {2p-1}~ \left( \mathcal{M}\cdot \mathbf{R}(pn,\text{spn}(S))-(\tilde{\mathcal{N}}\cdot \tilde{\mathcal{N}}^t )^{-1}\tilde{\mathcal{N}}\cdot \widetilde{\mathbf{R}}_2\right) & \text{otherwise}, \end{cases}
$$
where 
{\small $$
\widetilde{\mathbf{R}}_1 \!=\! \left( \begin{array}{c} \\ \mathbf{R}(pn, \spn(S)) \\\\ r(n, \spn(T))-r(n, \spn(\tilde{T}))\end{array}\right),  \ \widetilde{\mathbf{R}}_2 \!=\!\left( \begin{array}{c} \\ \mathbf{R}(pn, \spn(S)) + \mathbf{R}(p^3 n , \spn(S)) \\\\ (2p-1)(r(n, \spn(\tilde{T}))-r(n, \spn({T})))\end{array}\right).
$$}
\end{thm}

\begin{proof}
The proof is similar to that of Theorem \ref{rel0O}. First, assume that $\mathfrak G_{L,p}(2)$ is of $O$-type. Since the rank of $\mathcal N$ is $u$, we may define  $\mathcal Z=(\mathcal{N}\cdot \mathcal{N}^t)^{-1}\mathcal{N}$.
From the equation (4.3), we have
\begin{equation}
 \mathbf{R}(n,\spn(T))=\mathcal Z\left(\mathbf{R}(pn,spn(S))+\mathbf{R}^{\sharp}\left(\frac n p, \spn(\lambda_p (S))\right)\right),
\end{equation}
and 
\begin{equation}
 \mathbf{R}(p^2n,\spn(T))=\mathcal Z\left(\mathbf{R}(p^3n,spn(S))+\mathbf{R}^{\sharp}\left(pn, \spn(\lambda_p (S))\right)\right).
\end{equation}
If $(\Gamma_{p,1}(S)^{\frac 1 p},\Gamma_{p,2}(S)^{\frac 1 p}) \simeq ( T_1, T_2)$, then 
$$
(\Gamma_{p,1}(\lambda_p(S))^{\frac 1 p}, \Gamma_{p,2}(\lambda_p(S))^{\frac 1 p}) \simeq (\lambda_p(T_1),\lambda_p(T_2)).
$$ 
Hence  we have
\begin{equation}
\mathbf{R}^{\sharp}(pn,\spn(\lambda_p(S))) = \mathcal{N}^t\cdot \mathbf{R}^{\sharp}(n,\spn(\lambda_p(T)))-\mathbf{R}^{\sharp}(n,\spn(\lambda_p^2(S))),
\end{equation}
that is,
\begin{equation}
\mathbf{R}^{\sharp}(n,\spn(\lambda_p(T)))=\mathcal Z(\mathbf{R}^{\sharp}(pn,\spn(\lambda_p(S)))+\mathbf{R}^{\sharp}(n,\spn(\lambda_p^2(S))).
\end{equation}
By Proposition \ref{prop2}, we also have
\begin{equation}
\mathbf{R}(p^2 n,\spn(T)) + 2p~ \mathbf{R}(n,\spn(T)) = \mathcal{M} \cdot \mathbf{R}(pn,\spn(S)) + \mathbf{R}^{\sharp}(n,\spn(\lambda_p(T))).
\end{equation} 
If $n$ is not divisible by $p$, then (4.11) comes directly  from (4.12). Assume that $n$ is divisible by $p$. Since 
$\lambda_p^3(S) \simeq \lambda_p(S)$, we have 
\begin{equation}
\mathbf{R}^{\sharp}\left(\frac n p, \spn(\lambda_p (S))\right)=\mathbf{R}^{\sharp}(n,\spn(\lambda_p^2(S))).
\end{equation}
Therefore, the theorem follows from equations $(4.12), (4.13), (4.15)$ and $(4.16)$.

If we replace $\mathcal{N}$ by $\tilde{\mathcal{N}}$, then the proof of the case when $\mathfrak G_{L,p}(2)$ is of $E$-type is quite similar to this. 
\end{proof}


\begin{exam}
Let $p = 3$ and let $L=\langle 1,1,2\rangle$. Then $T=\langle 1,2,9\rangle \in \mathcal G_{L,p}(2)$ and  $S_1=\langle 1,2,27\rangle \in \mathcal G_{L,p}(3)$. In fact,  the graph $\mathfrak G_{L,p}(2)$ is of $O$-type and
$$
\mathcal G_{L,p}(3)/\sim=
\Bigg\{ S_1
, ~ S_2= \begin{pmatrix} 3&1&1\\1&4&2\\1&2&6\end{pmatrix},~ S_3=  \begin{pmatrix} 1&0&0\\0&5&1\\0&1&11\end{pmatrix},~ S_4=\begin{pmatrix} 2&0&0\\0&4&1\\0&1&7\end{pmatrix}\Bigg\}.
$$
In this case, one may easily check that there are no rational numbers $a_i$ and $b_i$ satisfying the equation
$$
r(n,T) = \sum_{i=1}^{4} a_i\cdot r(3n, S_i) + \sum_{i=1}^{4} b_i\cdot r(27n,S_i) \quad \text{for any integer $n$}.
$$ 
\end{exam}


Finally, assume that $m \ge 3$. Let $T \in \mathcal G_{L,p}(m)$ and $S\in \mathcal G_{L,p}(m+1)$ be $\z$-lattices such that $r(T^p, S) \ne 0$. We additionally assume that $\mathfrak G_{L,p}(m)$ is of $O$-type. Recall that $\mathcal{M} = \left( \frac{r(T_i^p,S_j)}{o(T_i)} \right)$ and $\mathcal{N} = \left( \frac{r(T_i^p,S_j)}{o(S_j)} \right)$.  We define $\mathcal Z=(\mathcal{N}\mathcal N^t)^{-1}\mathcal N$.

\begin{thm}\label{diff} Under the assumptions given above, if $n$ is not divisible by $p$, then 
$$
\mathbf{R}(n,\spn(T)) = \mathcal{Z}\left(\mathbf{R}(pn,\spn(S))\right)\ \  \text{and} \ \
\mathbf{R}(pn,\spn(T)) =\mathcal{M} \cdot \mathbf{R}(n,\spn(S)).
$$
For an arbitrary integer $n$, we have
{\small $$ 
\aligned 
&p \mathbf{R}(p^2 n,\spn(T))-p^2 \mathbf{R}(n,\spn(T))\\
&= \mathcal{Z} \left( 2p \mathbf{R}(p^3n,\spn(S)) + p^2 \mathbf{R}(pn,\spn(S))+\mathbf{R}^{\flat}(pn, \spn(S))\right)-p \mathcal{M}\cdot \mathbf{R}(pn,\spn(S)),
 \endaligned
$$}
where {\small $$\mathbf{R}^{\flat}(pn, \spn(S))=\left(\frac{o(\lambda_p(S_1))}{o(S_1)}r(pn,\gen_p^{\lambda_p(S_1)}(S_1)),\dots,\frac{o(\lambda_p(S_v))}{o(S_v)}r(pn,\gen_p^{\lambda_p(S_v)}(S_v))\right)^t.$$}
\end{thm}

 \begin{proof}  
     By Propositions \ref{prop1} and \ref{prop2}, we  have
\begin{equation}
\mathbf{R}(pn,\spn(S)) = \mathcal{N}^t \cdot \mathbf{R}(n,\spn(T)) - \mathbf{R}^{\sharp}\left(\frac np, \spn(\lambda_p(S))\right),
\end{equation}
and 
\begin{equation}
\aligned
\mathbf{R}(pn,\spn(T)) &=\mathcal{M} \cdot \mathbf{R}(n,\spn(S))\\  
     &+ \mathbf{R}^{\sharp}\left(\displaystyle\frac np, \spn(\lambda_p(T))\right)-2p \cdot \mathbf{R}\left(\displaystyle\frac np , \spn(T)\right).
     \endaligned
\end{equation}
The first two equations follow directly from (4.18) and (4.19). 

Now by applying $\lambda_p$-transformation to the equation (4.18), we also have
\begin{equation}
\mathbf{R}^{\sharp}(pn,\spn(\lambda_p(S))) = \mathcal{N}^t \cdot \mathbf{R}^{\sharp}(n,\spn(\lambda_p(T))) - \mathbf{R}^{\sharp}\left(\frac np, \spn(\lambda_p^2(S))\right).
\end{equation}
Our final ingredient is the following equation which is directly obtained  from Proposition \ref{genTL}:
\begin{equation}
\aligned  p\mathbf{R}^{\sharp}(pn,\spn(\lambda_p(S))) &+ p^2 \mathbf{R}^{\sharp}\left(\frac np, \spn(\lambda_p(S))\right) - p \mathbf{R}^{\sharp}\left(\frac np, \spn(\lambda_p^2(S))\right)\\  &=  \mathbf{R}^{\flat}(pn, \spn(S)).\endaligned
\end{equation}
 By multiplying $\mathcal{Z}$ to (4.18), we have 
\begin{displaymath}
\mathbf{R}(n,\spn(T)) = \mathcal{Z}\left(\mathbf{R}(pn,\spn(S)) + \mathbf{R}^{\sharp}\left(\frac np, \spn(\lambda_p(S))\right)\right).
\end{displaymath}
Hence we have
{\small $$
\aligned 
2p\mathbf{R}(p^2 n,\spn(T))&+p^2 \mathbf{R}(n,\spn(T))=2p\mathcal{Z} \left(\mathbf{R}(p^3 n,\spn(S)) +\mathbf{R}^{\sharp}(pn,\spn(\lambda_p(S)))\right) \\
          &\qquad \qquad+p^2 \mathcal Z \left(\mathbf{R}(pn, \spn(S)) + \mathbf{R}^{\sharp}\left(\frac np, \spn(\lambda_p(S))\right)\right).
          \endaligned
$$}
On the other hand, by combining (4.19) and (4.20), we have
$$
\aligned 
\mathbf{R}(p^2 n, \spn(T))+ &2p \mathbf{R}(n,\spn(T))-\mathcal{M} \cdot \mathbf{R}(pn,\spn(S)) \\
                           &=\mathcal Z \left(\mathbf{R}^{\sharp}(pn, \spn(\lambda_p (S))) + \mathbf{R}^{\sharp}\left(\frac np, \spn(\lambda_p^2(S))\right)\right).
\endaligned
$$ 
The theorem follows from the above two equations and (4.21). 
\end{proof}


\begin{thebibliography}{abcd}

\bibitem {a}  F.  Andrianov,  {\em Clifford algebras and the Shimura lift for theta series}, St. Petersburg Math. J.  \textbf{12}(2001),  51–-81. 

\bibitem {bh} J. W. Benham and J. S. Hsia, {\em On spinor exceptional representations}, Nagoya Math. J. \textbf{87}(1982), 247--260.

\bibitem{co0} W. K. Chan and B.-K. Oh, {\em Finiteness theorems for positive definite $n$-regular quadratic forms},  Trans. Amer. Math. Soc. {\bf 355}(2003), 2385--2396.

\bibitem {co} W. K. Chan and B.-K. Oh, {\em Class numbers of ternary quadratic forms}, J. Number Theory \textbf{135}(2014), 
221--261.





\bibitem {e} M. Eichler, {\em Quadratische Formen und orthogonale Gruppen}, Grundlehren Math. Wiss., vol. 63, Springer-Verlag, Berlin-New York, 1974.





\bibitem{k} M. Kneser, {\em Darstellungsma\ss e  indefiniter quadratischer Formen}, (German) Math. Z. \textbf{77}(1961), 188--194.
\bibitem {ki} Y. Kitaoka, {\em Arithmetic of quadratic forms}, Cambridge University Press, 1993.




\bibitem  {om2} O. T. O'Meara, {\em Introduction to quadratic forms}, Springer Verlag,
New York, 1963.



\bibitem {sp1} R. Schulze-Pillot, {\em Darstellung durch definite ternare quadratische Formen und das Bruhat-
Tits-Gebaude der Spingruppe}, Dissertation U, G\"ottingen 1979.

\bibitem{sp2} R. Schulze-Pillot, {\em Darstellungsma\ss e von Spinorgeschlechtern tern\"arer quadratischer Formen}, J. Reine Angew. Math. \textbf{352}(1984), 114-132.

\bibitem{sp3} R. Schulze-Pillot, {\em Thetareihen positiv definiter quadratischer Formen}, Invent. Math. \textbf{75}(1984), 283-229.



\bibitem {w} A. Weil, {\em Sur la th\'eorie des formes quadratiques}, Centre Belge Rech. Math., Colloque Theor. Groupes algebr., Bruxelles 1962, 9--22.

\end{thebibliography}
\end{document}